\documentclass[reqno]{amsart}

\usepackage[utf8]{inputenc}
\usepackage[T2A,T1]{fontenc}

\usepackage{lmodern}
\usepackage{amssymb}
\usepackage{amsmath}
\usepackage{mathrsfs}
\usepackage{stmaryrd}
\usepackage{enumitem}
\usepackage[olditem,oldenum]{paralist}

\usepackage{graphics}

\usepackage[all,2cell]{xy} 
\UseAllTwocells

\usepackage{dsfont}

\DeclareSymbolFont{EulerScripta}{U}{euf}{m}{n}
\SetSymbolFont{EulerScripta}{bold}{U}{euf}{b}{n} 
\DeclareSymbolFontAlphabet\matheufm{EulerScripta}

\DeclareSymbolFont{EulerScriptb}{U}{eur}{m}{n}
\SetSymbolFont{EulerScriptb}{bold}{U}{eur}{b}{n} 
\DeclareSymbolFontAlphabet\matheurm{EulerScriptb}

\DeclareSymbolFont{EulerScriptc}{U}{eus}{m}{n}
\SetSymbolFont{EulerScriptc}{bold}{U}{eus}{b}{n} 
\DeclareSymbolFontAlphabet\matheusm{EulerScriptc}
\newcommand\eusm{\matheusm}

\usepackage{xspace}

\usepackage{comment}

\newcommand{\theoremref}[1]{\hyperref[#1]{Theorem~\ref*{#1}}}
\newcommand{\lemmaref}[1]{\hyperref[#1]{Lemma~\ref*{#1}}}
\newcommand{\remarkref}[1]{\hyperref[#1]{Remark~\ref*{#1}}}
\newcommand{\sectionref}[1]{\hyperref[#1]{Section~\ref*{#1}}}
\newcommand{\definitionref}[1]{\hyperref[#1]{Definition~\ref*{#1}}}
\newcommand{\propositionref}[1]{\hyperref[#1]{Proposition~\ref*{#1}}}
\newcommand{\conjectureref}[1]{\hyperref[#1]{Conjecture~\ref*{#1}}}
\newcommand{\corollaryref}[1]{\hyperref[#1]{Corollary~\ref*{#1}}}
\newcommand{\exampleref}[1]{\hyperref[#1]{Example~\ref*{#1}}}
\newcommand{\exerciseref}[1]{\hyperref[#1]{Exercise~\ref*{#1}}}
\renewcommand{\eqref}[1]{\hyperref[#1]{(\ref*{#1})}}
\newcommand{\pararef}[1]{\hyperref[#1]{\S\ref*{#1}}}
\newcommand{\enumref}[1]{\hyperref[#1]{{\itshape{(\ref*{#1})}}}}
\newcommand{\propositionitemref}[2]{\hyperref[#1]{Proposition~\ref*{#1}-{\itshape{(\ref*{#2})}}}}

\theoremstyle{plain}
\newtheorem{theo}{Theorem}

\newtheorem{prop}[theo]{Proposition}
\newtheorem{lemm}[theo]{Lemma}
\newtheorem{coro}[theo]{Corollary}
\newtheorem{ques}[theo]{Question}

\newtheorem*{theo*}{Theorem}

\theoremstyle{definition}
\newtheorem{defi}[theo]{Definition}

\theoremstyle{remark}
\newtheorem{rema}[theo]{Remark}

\renewcommand{\ge}{\geqslant}
\renewcommand{\le}{\leqslant}

\newcommand{\ol}{\overline}

\newcommand{\red}{\mathrm{red}}
\newcommand{\bbG}{\mathbf{G}}

\newcommand{\xra}{\xrightarrow}
\newcommand{\op}{\mathrm{op}}
\newcommand{\scr}{\mathscr}
\newcommand{\ra}{\rightarrow}
\newcommand{\Id}{\mathrm{Id}}

\newcommand{\Hom}{\mathrm{Hom}}

\newcommand{\Osheaf}{\mathscr{O}}
\newcommand{\bbQ}{\mathbb{Q}}

\newcommand{\bbZ}{\mathbb{Z}}
\newcommand{\bbC}{\mathbb{C}}

\newcommand{\Spec}{\mathop{\mathrm{Spec}}}
\newcommand{\Db}{\mathrm{D}^{\mathrm{b}}}

\newcommand{\Gr}{\mathrm{Gr}}
\newcommand{\Ker}{\mathop{\mathrm{Ker}}\nolimits}
\newcommand{\Coker}{\mathop{\mathrm{Coker}}\nolimits}
\newcommand{\Coimg}{\mathop{\mathrm{Coim}}\nolimits}
\newcommand{\Img}{\mathop{\mathrm{Im}}\nolimits}

\renewcommand{\mod}{\mathop{\mathsf{mod}}\nolimits}

\renewcommand{\Vec}{\mathbf{Vec}}
\newcommand{\tVec}{\widetilde{\mathbf{Vec}}}

\newcommand{\sM}{\mathscr{M}}

\newcommand{\id}{\mathrm{id}}
\newcommand{\bH}{\mathbf{H}}

\newcommand{\Alb}{\mathrm{Alb}}
\newcommand{\Jac}{\mathrm{Jac}}
\newcommand{\Tr}{\mathrm{Tr}}

\newcommand{\gr}{\mathrm{Gr}}

\newcommand{\et}{{\mathrm{\acute{e}t}}}

\newcommand{\ab}{\mathrm{ab}}

\newcommand{\Ext}{\mathrm{Ext}}
\newcommand{\Pic}{\mathrm{Pic}}
\newcommand{\Div}{\mathrm{Div}}
\newcommand{\Lie}{\mathop{\mathrm{Lie}}\nolimits}

\newcommand{\LM}{\mathsf{LM}}

\newcommand{\Del}{\mathrm{Del}}
\newcommand{\Lau}{\mathrm{Lau}}

\newcommand{\sa}{\mathrm{sa}}

\newcommand{\tunMot}{{}^t\kern-3pt\mathscr M}

\newcommand{\MHSM}{\mathbf{MHSM}}
\newcommand{\MHS}{\mathbf{MHS}}
\newcommand{\EHS}{\mathbf{EHS}}
\newcommand{\FHS}{\mathbf{FHS}}

\newcommand{\add}{\mathrm{add}}
\renewcommand{\iff}{\mathrm{inf}}
\renewcommand{\inf}{\mathrm{inf}}

\newcommand{\Cone}{\mathrm{Cone}}
\newcommand{\Tot}{\mathrm{Tot}}
\newcommand{\Res}{\mathrm{Res}}
\newcommand{\ord}{\mathrm{ord}}

\newcommand{\fr}{\mathrm{fr}}
\newcommand{\Tor}{\mathrm{Tor}}

\newcommand{\Dbc}{\mathrm{D}^{\mathrm{b}}_{\mathrm{c}}}
\newcommand{\bbD}{\mathbb{D}}
\newcommand{\MHM}{\mathbf{MHM}}
\newcommand{\rat}{\mathrm{rat}}

\newcommand{\BS}{\mathsf{BS}}
\newcommand{\Homo}{\mathcal{H}om}

\makeatletter
\let\@@seccntformat\@seccntformat
\renewcommand*{\@seccntformat}[1]{%
  \expandafter\ifx\csname @seccntformat@#1\endcsname\relax
    \expandafter\@@seccntformat
  \else
    \expandafter
      \csname @seccntformat@#1\expandafter\endcsname
  \fi
    {#1}%
}
\newcommand*{\@seccntformat@section}[1]{%
  {{\csname the#1\endcsname.}}
}
\newcommand*{\@seccntformat@subsection}[1]{%
  {\bf{\csname the#1\endcsname.}}
}
\newcommand*{\@seccntformat@subsubsection}[1]{%
  {\bf{\csname the#1\endcsname.}}
}

\def\section{\@startsection{section}{1}%
  \z@{.7\linespacing\@plus\linespacing}{.5\linespacing}%
  {\normalfont\scshape\centering}}
\def\subsection{\@startsection{subsection}{2}%
  \z@{.5\linespacing\@plus.7\linespacing}{-.5em}%
  {\normalfont\bfseries}}
\makeatother

\usepackage[unicode,bookmarks]{hyperref}
\usepackage[usenames,dvipsnames]{xcolor}
\hypersetup{colorlinks=true,citecolor=NavyBlue,linkcolor=BrickRed,urlcolor=Green}

\begin{document}

\title{Mixed Hodge structures with modulus}
\author{Florian Ivorra}
\address{Institut de recherche math\'ematique de Rennes\\ UMR 6625 du CNRS\\ Universit\'e de Rennes 1\\
Campus de Beaulieu\\
35042 Rennes cedex (France)}
\email{florian.ivorra@univ-rennes1.fr}

\author{Takao Yamazaki}
\address{Institute of mathematics \\ Tohoku University\\ Aoba, Senda\"i\\
980-8578 (Japan)}
\email{ytakao@math.tohoku.ac.jp}



\begin{abstract}
We define a notion of mixed Hodge structure with modulus that generalizes the classical notion of mixed Hodge structure introduced by Deligne and the level one Hodge structures with additive parts introduced by Kato and Russell in their description of Albanese varieties with modulus. With modulus triples of any dimension we attach mixed Hodge structures with modulus. 
We combine this construction with an equivalence between 
the category of level one mixed Hodge structures with modulus
and the category of Laumon $1$-motives
to generalize Kato-Russell's Albanese varieties with modulus 
to $1$-motives.
\end{abstract}

\maketitle

\setcounter{tocdepth}{1}
\tableofcontents

\section{Introduction}

\subsection{Background}
Unlike K-theory, classical cohomology theories, such as Betti cohomology, \'etale cohomology or motivic cohomology (in particular Chow groups) are not able to distinguish a smooth variety from its nilpotent thickenings. This inability to detect nilpotence makes those cohomologies not the right tool to study non-homotopy invariant phenomena. One very important situation where these kind of phenomena occur, is at the boundary of a smooth variety. More precisely if $\overline{X}$ is a smooth proper variety and $D$ is an effective Cartier Divisor on $\overline{X}$, then $D$ can be seen as the non-reduced boundary at infinity of the smooth variety $X:=\overline{X}\setminus D$. Since the works of Rosenlicht and Serre (see \cite{MR918564}), it is known that the cohomology groups
that admit a geometrical interpretation in terms of Jacobian varieties or Albanese varieties, do admit generalizations able to see the non-reducedness of the boundary (unipotent groups appear in those generalized Jacobians).

In recent years, most of the developments, following the work of Bloch-Esnault \cite{BE03}, have focused on the algebraic cycle part of the story. In this work, we focus on the Hodge theoretic counterpart of these developments.

\subsection{Main results}
In the present paper, we introduce a notion of mixed Hodge structure with modulus (see \definitionref{defi:MHSM}), that generalizes the classical notion of mixed Hodge structure introduced by Deligne \cite{MR0498551}.  It is closely related to the notion of enriched Hodge structure intoduced by Bloch-Srinivas \cite{MR1940668} and the notion of formal Hodge structure introduced by Barbieri-Viale \cite{MR2318642} and studied by Mazzari \cite{MR2782612}. However the relationship is not trivial, see \sectionref{sect:comparison}.
Our main results are summerized as follows:

\begin{enumerate}
\item 
The category $\MHSM$ of mixed Hodge structures with modulus is Abelian.
It contains the usual category of mixed Hodge structures $\MHS$ as a full subcategory. Duality and Tate twists extend to mixed Hodge structures with modulus.
\item The category $\MHSM$ of mixed Hodge structures with modulus contains a full subcategory $\MHSM_1$ which is equivalent to the category of Laumon 1-motives (the duality functor on mixed Hodge structures with modulus corresponding via this equivalence to Cartier duality).
\item
Given a smooth proper variety $X$
and two effective simple normal crossing divisors $Y, Z$ on $X$
such that $|Y| \cap |Z|=\emptyset$,
we associate functorially an object
$\eusm H^n(X, Y, Z)$ of $\MHSM$ for each $n \in \bbZ$.
Its underlying mixed Hodge structure is 
given by the relative cohomology $H^n(X \setminus Z, Y, \bbZ)$.
\item
For $(X, Y, Z)$ as above,
if further $X$ is equidimensional of dimension $d$,
then we have a duality theorem
($\fr$ denotes the free part, see \pararef{sect:free}):
\begin{equation*}
\eusm H^n(X,Y,Z)^\vee \cong \eusm H^{2d-n}(X,Z,Y)(d)_\fr.
\end{equation*}
\end{enumerate}


Our construction is closer to 
Kato-Russell's category $\mathcal H_1$ from \cite{MR2985516}.
It is also motivated by the recent developments
of the theory of algebraic cycles with modulus
(such as additive Chow group \cite{BE03}, 
higher Chow groups with modulus  \cite{BS},
and Suslin homology with modulus \cite{RY,KSY}),
to which our theory might be considered as 
the Hodge theoretic counterpart.
We hope to study their relationship in a future work.
We also leave as a future problem
a construction of an object of $\MHSM$
that overlays Deligne's mixed Hodge structure on $H^n(X, \bbZ)$
for non-proper $X$.

\subsection{Application to Albanese $1$-motives}
For a pair $(X,Y)$ consisting of a smooth proper variety $X$
and an effective divisor $Y$ on $X$,
Kato and Russell constructed in \cite{MR2985516}
the Albanese variety with modulus $\Alb^{KR}(X, Y)$
as a higher dimensional analogue of the
the generalized Jacobian variety of Rosenlicht-Serre.
Our theory yields an extension of their construction to $1$-motives.
This goes as follows.

Given a triple $(X, Y, Z)$ as in (3), (4) above,
it is easy to see that
the mixed Hodge structure with modulus $\eusm H^{2d-1}(X,Y,Z)(d)_\fr$ belongs to the subcategory $\MHSM_1$.
Therefore, it produces a Laumon $1$-motive 
$\Alb(X, Y, Z)$ corresponding to $\eusm H^{2d-1}(X,Y,Z)(d)_\fr$
under the equivalence (2) above.
When $Z=\emptyset$, it turns out that 
$\Alb(X, Y, \emptyset)=[0 \to \Alb^{KR}(X, Y)]$. 
When $d=1$, $\Alb(X, Y, Z)$ agrees with
the Laumon $1$-motive  $\LM(X, Y, Z)$
constructed in \cite[Definition 25]{IY}.

\subsection{Organization of the paper}
The definition of the mixed Hodge structures with modulus
is given in \sectionref{sect:MHSM}.
Its connection with Laumon $1$-motives is studied in \sectionref{sect:laumon}.
We construct $\eusm H^n(X, Y, Z)$ in \sectionref{sec:Geo},
and prove the duality in \sectionref{sec:dual}.
In \sectionref{sect:pic-alb} we construct Albanese $1$-motives.
The last \sectionref{sect:comparison} is devoted to comparison 
with $\EHS$ and $\FHS_n$.

\subsection*{Acknowledgements}
Part of this work was done during a visit of the first author at the Tohoku University and a visit of the second author at the University of Rennes 1. The hospitality of those institutions is gratefully appreciated.
This work is supported by JSPS KAKENHI Grant (15K04773) and the University of Rennes~1.


\section{Mixed Hodge structures with modulus}\label{sect:MHSM}

\subsection{}\label{sec:def-MHSM}
Let $\Vec_\bbC$ be the category of finite dimensional $\bbC$-vector spaces.
Let $\mathbf{Z}$ be the category associated to the ordered set $\bbZ$ and consider the category $\mathbf{Z}^\op \Vec_\bbC$ of 
functors $\mathbf{Z}^\op \to \Vec_\bbC$,
that is,
sequences of $\Vec_\bbC$
(which may be neither injective nor surjective and may not form a complex)
\begin{equation}\label{eq:vec-bul}
\cdots\ra V^k\xra{\tau^k_V}V^{k-1}\xra{\tau^{k-1}_V}V^{k-1}\ra\cdots.
\end{equation}
We denote by $\Vec^\bullet_\bbC$ the strictly full subcategory of $\mathbf{Z}^\op\Vec_\bbC$ formed by the objects $V^\bullet$ such that $V^k=0$ for all but finitely elements $k\in\bbZ$.

We denote by $\MHS$ the category of mixed Hodge structures. 
For an object $H$ of $\MHS$,
we denote by $H_\bbZ$ its underlying a finitely generated $\bbZ$-module,
by $W_\bullet H_\bbQ$ the weight filtration 
on $H_\bbQ:=H_\bbZ \otimes \bbQ$,
and 
by $F^\bullet H_\bbC$ 
the Hodge filtration on $H_\bbC:=H_\bbZ \otimes \bbC$.

Given an object $\eusm H:=(H,H^\bullet_\add,H^\bullet_\iff)$ in the product category $\MHS\times\Vec_\bbC^\bullet\times\Vec^\bullet_\bbC$, we set
\begin{align*}
\eusm H^k&:= H_\bbC\oplus H^k_\add\oplus H^k_\iff,\\
\tau ^k& := \Id\oplus \tau_\add^k\oplus \tau^k_\iff : \eusm H^k \to \eusm H^{k-1},
\end{align*}
where 
$\tau^k_\add:H^k_\add\ra H^{k-1}_\add$ and 
$\tau^k_\iff:H^k_\iff\ra H^{k-1}_\iff$ 
are the structural maps.  
\begin{defi}\label{defi:MHSM}
A mixed Hodge structure with modulus is a tuple 
\[\matheusm{H}:=(H,H^\bullet_\add,H^\bullet_\iff,\eusm F^\bullet),\] 
consisting of a mixed Hodge structure $H$, two objects $H^\bullet_\add,H^\bullet_\iff$ in $\Vec^\bullet_\bbC$, and for every $k\in\bbZ$ a linear subspace $\matheusm F^k$ of $\eusm H^k$ such that the following conditions are satisfied  for every $k\in\bbZ$:
\begin{enumerate}
\item[{\bf{(\ref{defi:MHSM}-a})}] $\tau^k(\eusm F^k)\subseteq \eusm F^{k-1}$; 
\item[{\bf{(\ref{defi:MHSM}-b})}] an element $x\in H_\bbC$ is in $F^kH_\bbC$ if and only if there exists $v\in H^k_\add$ such that $x+v\in \eusm F^k$;
\item[{\bf{(\ref{defi:MHSM}-c})}] $\eusm H^k=\eusm F^k+H_\bbC+H_\add^k$;
\item[{\bf{(\ref{defi:MHSM}-d})}] $H_\add^k\cap \eusm F^k=0$.
\end{enumerate}
By abuse of terminology,
we call $\eusm F^\bullet$ the Hodge filtration on $\eusm H$.
A morphism between two mixed Hodge structures with modulus
is a morphism of
$\MHS \times \Vec_\bbC^\bullet \times \Vec_\bbC^\bullet$ 
that respects Hodge filtrations.
The category of mixed Hodge structures with modulus
is denoted by $\MHSM$. A mixed Hodge structure with modulus is said to be polarizable if its underlying mixed Hodge structure is.
\end{defi}

\begin{rema}
The conditions {\bf{(\ref{defi:MHSM}-c})} and {\bf{(\ref{defi:MHSM}-d})} can be rewritten in a more symmetric way (in the sense of the opposite category). Indeed, they are equivalent to requiring that the linear map $\eusm F^k\hookrightarrow \eusm H^k\twoheadrightarrow H_\iff^k$ is surjective and the linear map
$H_\add^k\hookrightarrow \eusm H^k\twoheadrightarrow \eusm H^k/\eusm F^k$
is injective. 
\end{rema}
Our \definitionref{defi:MHSM} is motivated by preceding works 
\cite{MR1940668}, \cite{MR2318642}, \cite{MR2782612}, \cite{MR2985516},
as well as 
the geometric example described in \sectionref{sec:Geo}.

\subsection{}\label{sec:h-inf}
Let 
$\matheusm{H}=(H,H^\bullet_\add,H^\bullet_\iff,\eusm F^\bullet)$
be an object of $\MHSM$.
For each integer $k$, put 
\begin{align*}
\eusm H^k_\iff &:= H_\bbC\oplus H^k_\iff,
\\
\eusm F^k_\iff &:= 
\{ x \in \eusm H^k_\iff  
\mid x+v \in \eusm F^k ~\text{for some}~v \in H^k_\add \}
=\Img(\eusm F^k \subset \eusm H^k \twoheadrightarrow \eusm H^k_\inf).
\end{align*}
This definition and 
the condition {\bf{(\ref{defi:MHSM}-d})} implies that
the projection map $\eusm H^k \to \eusm H_\iff^k$
restricts to an isomorphism $\eusm F^k \cong \eusm F^k_\iff$,
and we get a commutative diagram
\begin{equation}\label{eq:fund-exseq1}
\xymatrix@R=.35cm@C=.35cm{{} & & {0}\ar[d] & {0}\ar[d] & {} \\
 & & {\eusm F^k}\ar[r]^-{\simeq}\ar[d] & {\eusm F^k_\iff}\ar[d] & {}\\
{0}\ar[r] &{H^k_\add}\ar@{=}[d]\ar[r] & {\eusm H^k}\ar[r]\ar[d] & {\eusm H^k_\iff}\ar[r]\ar[d] & {0}\\
{0}\ar[r] & {H^k_\add}\ar[r] & {\eusm H^k/\eusm F^k}\ar[r]\ar[d] & {\eusm H^k_\iff/\eusm F^k_\iff}\ar[d]\ar[r] & {0} \\
{} & & {0} & {0} & {} }
\end{equation}
made of short exact sequences.
It follows from this diagram and {\bf{(\ref{defi:MHSM}-c})}
that $\eusm H_\iff^k=\eusm F^k_\iff+H_\bbC$.
Therefore we find that
\begin{equation}\label{eq:h_inf}
\eusm H_\iff:=(H, 0, H^\bullet_\iff, \eusm F^\bullet_\iff)
\end{equation}
is an object of $\MHSM$.
We obtain a functor 
\[ \pi_\iff : \MHSM \to \MHSM_\iff, 
\quad \pi_\iff(\eusm H)=\eusm H_\iff,
\]
where $\MHSM_\iff$ is the full subcategory of $\MHSM$ 
consisting of 
$(H,H^\bullet_\add,H^\bullet_\iff,\eusm F^\bullet)$
such that $H^\bullet_\add$ is trivial.
This is a left adjoint of 
the inclusion functor $i_\iff : \MHSM_\iff \to \MHSM$.

\subsection{}\label{sec:h-add}
Similarly, 
for an object
$\matheusm{H}=(H,H^\bullet_\add,H^\bullet_\iff,\eusm F^\bullet)$
of $\MHSM$
put 
\[
\eusm H^k_\add := H_\bbC \oplus H^k_\add,
\quad
\eusm F^k_\add := 
\eusm F^k \cap \eusm H_\add^k
=\ker(\eusm H^k_\add \subset \eusm H^k \twoheadrightarrow \eusm H^k/\eusm F^k).
\]
This definition and 
the condition {\bf{(\ref{defi:MHSM}-c})} implies that
the inclusion map $\eusm H^k_\add \to \eusm H^k$
induces an isomorphism $\eusm H^k_\add/\eusm F^k_\add \cong \eusm H^k/\eusm F^k$
and we get a commutative diagram 
\begin{equation}\label{eq:fund-exseq2}
\xymatrix@R=.35cm@C=.35cm{{} & {0}\ar[d] & {0}\ar[d] & {} & {}\\
{0}\ar[r] & {\eusm F^k_\add}\ar[r]\ar[d] & {\eusm F^k}\ar[r]\ar[d] & {H_\iff^k}\ar[r]\ar@{=}[d] & {0}\\
{0}\ar[r] & {\eusm H^k_\add}\ar[r]\ar[d] & {\eusm H^k}\ar[r]\ar[d] & {H_\iff^k}\ar[r] & {0}\\
{} & {\eusm H^k_\add/\eusm F^k_\add}\ar[r]^-{\simeq}\ar[d] & {\eusm H^k/\eusm F^k}\ar[d] & {} & {}\\
{} & {0} & {0} & {} & {}}
\end{equation}
made of short exact sequences.
It follows from {\bf{(\ref{defi:MHSM}-d})}
that $H^k_\add\cap \eusm F_\add^k=0$,
and
\begin{equation}\label{eq:h_add}
\matheusm{H}_\add:=
(H, H^\bullet_\add, 0, \eusm F^\bullet_\add)
\end{equation}
is an object of $\MHSM$.
We obtain a functor 
\[ \pi_\add : \MHSM \to \MHSM_\add, 
\quad \pi_\add(\eusm H)=\eusm H_\add,
\]
where $\MHSM_\add$ is the full subcategory of $\MHSM$ 
consisting of 
$(H,H^\bullet_\add,H^\bullet_\iff,\eusm F^\bullet)$
such that $H^\bullet_\iff$ is trivial.
This is a right adjoint of 
the inclusion functor $i_\add : \MHSM_\add \to \MHSM$.

\subsection{}
We identify $\MHS$
with the intersection of $\MHSM_\iff$ and $\MHSM_\add$ in $\MHSM$.
Then $\pi_\iff$ and $\pi_\add$ restrict to 
\[ \pi_\iff^0 : \MHSM_\add \to \MHS,
\quad
   \pi_\add^0 : \MHSM_\iff \to \MHS,
\]
and they are left and right adjoints of the inclusion functors
\[ i_\iff^0 : \MHS \to \MHS_\add,
\quad
   i_\add^0 : \MHS \to \MHS_\iff,
\]
respectively (see \eqref{eq:functors} below).
Let
$\matheusm{H}=(H,H^\bullet_\add,H^\bullet_\iff,\eusm F^\bullet)$
be an object of $\MHSM$.
We have
$\pi_\add^0 \pi_\iff \eusm H =
 \pi_\iff^0 \pi_\add \eusm H = (H, 0, 0, F^\bullet H_\bbC)$.
We may apply the results
of \pararef{sec:h-add} and \pararef{sec:h-inf} 
to $\eusm H_\iff$ and $\eusm H_\add$ respectively, 
yielding  commutative diagrams 
\begin{equation}\label{eq:fund-exseq3}
\xymatrix@R=.35cm@C=.35cm{{} & {0}\ar[d] & {0}\ar[d] & {} & {}\\
{0}\ar[r] & {F^kH_\bbC}\ar[r]\ar[d] & {\eusm F^k_\iff}\ar[r]\ar[d] & {H_\iff^k}\ar[r]\ar@{=}[d] & {0}\\
{0}\ar[r] & {H_\bbC}\ar[r]\ar[d] & {\eusm H^k_\iff}\ar[r]\ar[d] & {H_\iff^k}\ar[r] & {0}\\
{} & {H_\bbC/F^kH_\bbC}\ar[r]^-{\simeq}\ar[d] & {\eusm H^k_\iff/\eusm F^k_\iff}\ar[d] & {} & {}\\
{} & {0} & {0} & {} & {}}
\quad
\xymatrix@R=.35cm@C=.35cm{{} & & {0}\ar[d] & {0}\ar[d] & {} \\
 & & {\eusm F^k_\add}\ar[r]^{\simeq}\ar[d] & {F^kH_\bbC}\ar[d] & {}\\
{0}\ar[r] &{H^k_\add}\ar@{=}[d]\ar[r] & {\eusm H^k_\add}\ar[r]\ar[d] & {H_\bbC}\ar[r]\ar[d] & {0}\\
{0}\ar[r] & {H^k_\add}\ar[r] & {\eusm H^k_\add/\eusm F^k_\add}\ar[r]\ar[d] & {H_\bbC/F^kH_\bbC}\ar[d]\ar[r] & {0} \\
{} & & {0} & {0} & {} }
\end{equation}
made of short exact sequences. 
In particular, there exists a unique $\bbC$-linear map $\eusm H^k/\eusm F^k\ra H_\bbC/F^kH_\bbC$ which makes the following diagram
\begin{equation}\label{eq:map-hf}
\xymatrix{{\eusm H^k/\eusm F^k}\ar[rd]\ar[r] & {\eusm H^k_\iff/\eusm F^k_\iff}\\
{\eusm H^k_\add/\eusm F^k_\add}\ar[r]\ar[u]^-{\simeq} & {H^k_\bbC/F^kH_\bbC}\ar[u]^-{\simeq}} 
\end{equation}
commute (the vertical maps are induced by inclusions and the horizontal ones by projections).

\subsection{} As the following proposition shows the category of mixed Hodge structures with modulus is Abelian.

\begin{prop}
\begin{enumerate}
\item 
Any morphism $f : \eusm H \to \eusm H'$ in $\MHSM$
is strict with respect to the Hodge filtration,
that is,
$f(\eusm F^k) = \eusm F'^k \cap f(\eusm H^k)$ for any $k$
where
$\eusm H:=(H,H^\bullet_\add,H^\bullet_\iff, \eusm F^\bullet)$ and
$\eusm H':=(H',{H'}^\bullet_\add,{H'}^\bullet_\iff, {\eusm F'}^\bullet)$.
\item 
The category $\MHSM$ is an Abelian category.
\end{enumerate}
\end{prop}

\begin{proof}
The category $\MHSM$ is an additive category that has kernels and cokernels. 
Let $\Img f$ be the kernel of the canonical morphism $\eusm H'\ra\Coker f$
and $\Coimg f$ the cokernel of the canonical morphism $\ker f \ra \eusm H$.
Note that $\Coimg f$ is the tuple 
\[(H_\bbZ/\Ker f_\bbZ, H^\bullet_\add/\Ker f^\bullet_\add,H^\bullet_\iff/\Ker f^\bullet_\iff,\eusm F^\bullet/(\eusm F^\bullet\cap \Ker f^\bullet)\]
while $\Img f $ is the tuple
\[(\Img f_\bbZ,\Img f^\bullet_\add,\Img f^\bullet_\iff,\eusm F'^\bullet\cap\Img f^\bullet).\]

Recall that for every object $\eusm H$ in $\MHSM$, 
we have a (functorial) exact sequence 
and an isomorphism 
from \eqref{eq:fund-exseq2}, \eqref{eq:fund-exseq3}
$$0\ra \eusm F^k_\add\ra\eusm F^k\ra H^k_\iff\ra 0,
\qquad \eusm F^k_\add\xra{\simeq}F^kH_\bbC.
$$
By applying these to both $\Coimg f$ and $\Img f$, 
we get a commutative diagram with exact rows
\[
\xymatrix{
0 \ar[r] &
F^kH_\bbC/(F^kH_\bbC\cap\Ker f_\bbC)\ar[r] \ar[d] &
\eusm F^k/\eusm F^k \cap \Ker f^k \ar[r] \ar[d] &
H_\inf^k/\Ker f_\inf^k \ar[r] \ar[d]_\cong & 0
\\
0 \ar[r] &
F^kH'_\bbC\cap\Img f_\bbC  \ar[r]  &
{\eusm F'}^k \cap f(\eusm H^k) \ar[r]  &
\Img f_\inf^k \ar[r]  & 0.
}
\]
Thus (1) is reduced to showing that 
the left vertical map
is an isomorphism. This follows from the fact that $\MHS$ is an Abelian category (that is every morphism of mixed Hodge structures is strict with respect to the Hodge filtration).
(2) follows from (1).
\end{proof}

%

\subsection{}
Let us consider functors
\begin{align*}
&\mu_\add : \MHSM_\add \to \Vec_\bbC^\bullet,
\quad
\mu_\add(H, H_\add^\bullet, 0, \eusm F^\bullet_\add) = H_\add^\bullet,
\\
&\mu_\iff : \MHSM_\iff \to \Vec_\bbC^\bullet,
\quad
\mu_\iff(H, 0, H_\iff^\bullet, \eusm F^\bullet_\iff) = H_\iff^\bullet,
\\
&j_\add : \Vec_\bbC^\bullet \to \MHSM_\add
\quad
j_\add(V^\bullet)=(0, V^\bullet, 0, 0),
\\
&j_\iff : \Vec_\bbC^\bullet \to \MHSM_\iff
\quad
j_\iff(V^\bullet)=(0, 0, V^\bullet, V^\bullet).
\end{align*}
Then $\mu_\iff$ is a left adjoint of $j_\iff$
and
$\mu_\add$ is a right adjoint of $j_\add$.
%
%
We summarize the functors we have introduced so far:
\begin{equation}
\label{eq:functors}
\xymatrix{
& & 
\MHSM 
\ar@/^1em/[dl]^(0.3){\pi_\add}
\ar@/_1em/[dr]_(0.3){\pi_\iff}
& &
\\
&
\MHSM_\add
\ar@/^1em/[ur]^(0.5){i_\add}
\ar@/^1em/[dl]^(0.5){\mu_\add}
\ar@/_1em/[dr]_(0.5){\pi_\iff^0}
& & 
\MHSM_\iff
\ar@/_1em/[ul]_(0.5){i_\iff}
\ar@/_1em/[dr]_(0.5){\mu_\iff}
\ar@/^1em/[dl]^(0.5){\pi_\add^0}
& 
\\
\Vec^\bullet_\bbC 
\ar@/^1em/[ur]^(0.5){j_\add}
& & 
\MHS 
\ar@/_1em/[ul]_(0.3){i_\iff^0}
\ar@/^1em/[ur]^(0.3){i_\add^0}
& & 
\Vec^\bullet_\bbC. 
\ar@/_1em/[ul]_(0.5){j_\iff}
}
\end{equation}

\begin{rema}\label{rem:3step-fil}
\begin{enumerate}
\item 
Let $\eusm H$ be an object in $\MHSM$. 
We write
$\eusm H_\inf=i_\inf \pi_\inf(\eusm H)$ 
and $\eusm H_\add=i_\add \pi_\add(\eusm H)$,
see \eqref{eq:h_inf}, \eqref{eq:h_add}.
Let us also abbreviate
$H_\iff=i_\iff j_\iff \mu_\iff \pi_\iff(\eusm H),
H_\add=i_\add j_\add \mu_\add \pi_\add(\eusm H)$
and
$H=\pi_\iff^0 \pi_\add(\eusm H)=\pi_\add^0 \pi_\inf (\eusm H)$.
Various (co)unit maps make a commutative diagram in $\MHSM$
\[
\xymatrix{
& & 
\eusm H
\ar[rd]
& &
\\
&
\eusm H_\add
\ar[rd] \ar[ru]
& & 
\eusm H_\iff
\ar[rd]
& 
\\
H_\add 
\ar[ru]
& & 
H
\ar[ru]
& & 
H_\iff,
}
\]
which produces the following (functorial) short exact sequences
\[0\ra \eusm H_\add\ra\eusm H\ra H_\iff\ra0,\quad 0\ra H_\add\ra\eusm H\ra \eusm H_\iff\ra0 \]
and 
\[0\ra H_\add\ra \eusm H_\add\ra H\ra 0,\quad 0\ra H\ra \eusm H_\iff\ra H_\iff\ra 0.
\]
\item 
Any morphism $f : \eusm H \to \eusm H'$ in $\MHSM$
is strict with respect to the filtration
\[ H_\add \subset \eusm H_\add \subset \eusm H,
\quad
   H_\add' \subset \eusm H_\add' \subset \eusm H',
\]
that is, 
$f(\eusm H) \cap \eusm H_\add' = f(\eusm H_\add),~
 f(\eusm H) \cap H_\add' = f(H_\add)$.
\end{enumerate}
\end{rema}

\subsection{}\label{sect:free}
Recall that a mixed Hodge structure $H$
is called free if $H_\bbZ$ is free as a $\bbZ$-module.
In this case it makes sense to define 
$W_k H_\bbZ := W_k H_\bbQ \cap H_\bbZ$.
For general $H$ we define
its free part by 
$H_\fr:=(H_\bbZ/H_{\bbZ, \Tor}, W_\bullet H_\bbQ, F^\bullet H_\bbC)$.
A mixed Hodge structure with modulus
$\eusm H=(H,H^\bullet_\add,H^\bullet_\iff,\eusm F^\bullet)$
is called free if $H$ is.
For general $\eusm H$ we define
its free part by
$\eusm H_\fr:=(H_\fr,H^\bullet_\add,H^\bullet_\iff,\eusm F^\bullet)$.

\subsection{} \label{para:duality}
Let $H$ be a mixed Hodge structure.
The dual mixed Hodge structure
$H^\vee=(H^\vee_\bbZ, W_\bullet H_\bbQ^\vee, F^\bullet H_\bbC^\vee)$
of $H$
is defined by 
\[ H^\vee_\bbZ = \Hom_{\bbZ}(H, \bbZ), \quad
W_k H_\bbQ^\vee = (H_\bbQ /W_{-1-k} H_\bbQ)^\vee,  \quad
F^k H_\bbC^\vee = (H_\bbC/F^{1-k} H_\bbC)^\vee,
\]
where $\vee$ on the right hand side denotes
linear dual
(see \cite[1.1.6]{MR0498551}, \cite[1.6.2]{MR602463}).
Let $\eusm H=(H,H^\bullet_\add,H^\bullet_\iff,\eusm F^\bullet)$ 
be an object in $\MHSM$. 
We define the dual $\eusm H^\vee$ of $\eusm H$ as the tuple
\[\eusm H^\vee:=(H^\vee,H^{\vee,\bullet}_\add,H^{\vee,\bullet}_\iff,\eusm F^{\vee,\bullet}).\]
Here $H^\vee$ is the dual of the mixed Hodge structure $H$ 
and for every $k\in \bbZ$
\[
H^{\vee,k}_\add:=(H^{1-k}_\iff)^\vee, \quad
H^{\vee,k}_\iff:=(H^{1-k}_\add)^\vee, \quad
\eusm F^{\vee,k}:=(\eusm H^{1-k}/\eusm F^{1-k})^\vee.
\]
It is straightforward to see 
that the tuple $\eusm H^\vee$ belongs to $\MHSM$.
By definition $\eusm H^\vee$ is always free,
and we have 
\begin{equation}\label{eq:free-dual}
\eusm H^\vee = (\eusm H_\fr)^\vee,
\quad
(\eusm H^\vee)^\vee \cong \eusm H_\fr.
\end{equation}


\subsection{}\label{sect:Tatetwist}
Let $m$ be an integer. 
Recall that the Tate twist $H(m)$ of 
a mixed Hodge structure $H$ is defined by 
\[ H(m)_\bbZ = (2 \pi i)^m H_\bbZ,~
W_kH(m)_\bbQ = (2 \pi i)^m W_{k+2m} H_\bbQ,~
F^kH(m)_\bbC = F^{k+m} H_\bbC.
\]
Let $\eusm H=(H,H^\bullet_\add,H^\bullet_\iff,\eusm F^\bullet)$ 
be an object in $\MHSM$. 
We define the Tate twist 
$\eusm H(m) 
= (H(m), H(m)_\add^\bullet, H(m)_\iff^\bullet, \eusm F(m)^\bullet)\in \MHSM$
of $\eusm H$
by
\[
H(m)_\add^k =H_\add^{k+m}, \quad
H(m)_\iff^k =H_\iff^{k+m},\quad
\eusm F(m)^k = \eusm F^{k+m}.
\]

\subsection{}\label{sect:functor-R}
Let $\mod(\bbZ)$ be the category of finitely generated 
Abelian groups.
There is a faithful exact functor
\begin{align*}
&R : \MHSM \to \mod(\bbZ) \times\Vec_\bbC,
\\
&R(H, H_\add^\bullet, H_\inf^\bullet, \eusm F^\bullet)
= \left(H_\bbZ,~
\bigoplus_{k \in \bbZ} \left(H_\add^k \oplus H_\inf^k \right) \right).
\end{align*}
\begin{rema}\label{rem:exact}
\begin{enumerate}
\item 
A sequence in $\MHSM$ is exact if and only if
its image by $R$ is exact in $\mod(\bbZ) \times \Vec_\bbC$.
\item 
For any object $\eusm H$ of $\MHSM$,
we have a canonical isomorphism
$R(\eusm H^\vee) \cong R(\eusm H)^\vee$,
where on the right hand side $\vee$ denotes
the dual functor given by 
$(A, V)^\vee=(\Hom_\bbZ(A, \bbZ),~ \Hom_\bbC(V, \bbC))$.
\end{enumerate}
\end{rema}

\section{Laumon $1$-motives}\label{sect:laumon}

In \cite[\S4.1]{MR2985516}, Kato and Russell have defined a category $\mathcal H_1$ which provides a Hodge theoretic description of the category $\sM_1^\Lau$ of Laumon 1-motives over $\bbC$ that extends Deligne's description 
\cite[\S10]{MR0498552}
of the category $\sM_1^\Del$ of Deligne 1-motives over $\bbC$ 
in terms of the full subcategory $\MHS_1$ of $\MHS$  (see \pararef{para:MHS1} for its definition). 
In this section we define a subcategory of 
$\MHSM_1$ of $\MHSM$ which is equivalent to $\mathcal H_1$,
yielding an equivalence between
$\MHSM_1$ and $\sM_1^\Lau$ (\corollaryref{cor:laumon}).
This will be used for our construction of 
Picard and Albanese $1$-motives in \sectionref{sect:pic-alb}.
There is another Hodge theoretic description of $\sM_1^\Lau$,
due to Barbieri-Viale \cite{MR2318642},
in terms of the category $\FHS_1^\fr$ of 
torsion free formal Hodge structures of level $\le 1$.
As is explained in \cite[\S4.6]{MR2985516},
two categories $\mathcal H_1$ and $\FHS_1^\fr$ are equivalent. 

\subsection{}\label{para:MHS1}
Let $\MHS_1$ be the full subcategory of $\MHS$ 
formed by the free mixed Hodge structures 
of Hodge type 
\[\{(0,0),(-1,0),(0,-1),(-1,-1))\}\]
such that $\Gr_{-1}^W$ is polarizable
(see \cite[Construction (10.1.3)]{MR0498552}). 
Recall that such a mixed Hodge structure is simply a free Abelian group of finite rank $H_\bbZ$ with two filtrations (on $H_\bbQ:=\bbQ\otimes_\bbZ H_\bbZ$ and $H_\bbC=\bbC\otimes_\bbQ H_\bbQ$)
\[0=W_{-3}H_\bbQ\subseteq W_{-2}H_\bbQ\subseteq W_{-1}H_\bbQ\subseteq W_0H_\bbQ=H_\bbQ \]
\[0=F^{1}H_\bbC\subseteq F^0H_\bbC\subseteq F^{-1}H_\bbC=H_\bbC\]
such that $F^0\Gr^{W}_0H_\bbC=\Gr^W_0H_\bbC $ (that is $F^0H_\bbC+W_{-1}H_\bbC=H_\bbC$), $F^0W_{-2}H_\bbC=0$ and $\Gr^{W}_{-1}H_\bbZ$ is a polarizable pure Hodge structure of weight $-1$.
(See \pararef{sect:free} for $W_\bullet H_\bbZ$.)

\subsection{}\label{subsec:KatoRussell} Let $\mathcal H_1$ be the Abelian category defined by Kato and Russell in \cite[\S4.1]{MR2985516}. Recall that an object in $\mathcal H_1$ is a pair $(H_\bbZ,H_V)$ consisting of a free Abelian group of finite rank $H_\bbZ$ and a $\bbC$-vector space $H_V$ together with
\begin{enumerate}
\item[(a)] two $2$-step filtrations (called weight filtrations)
\[0=W_{-3}H_\bbQ\subseteq W_{-2}H_\bbQ\subseteq W_{-1}H_\bbQ\subseteq W_0H_\bbQ=H_\bbQ \]
\[0=W_{-3}H_V\subseteq W_{-2}H_V\subseteq W_{-1}H_V\subseteq W_0H_V=H_V \]
on $H_\bbQ:=\bbQ\otimes_\bbZ H_\bbZ$ and $H_V$;
\item[(b)] a $1$-step filtration (called Hodge filtration)
\[0=F^1H_V\subseteq F^0H_V\subseteq F^{-1}H_V=H_V; \]
\item[(c)] two $\bbC$-linear maps $a:H_\bbC:=\bbC\otimes_\bbQ H_\bbQ\ra H_V$ and $b:H_V\ra H_\bbC$ which are compatible with the weight filtrations (that is, $a$ maps $W_kH_\bbC:=\bbC\otimes_\bbQ W_kH_\bbQ$ to $W_kH_V$ and $b$ maps $W_kH_V$ to $W_kH_\bbC$) and such that $b\circ a=\Id$;
\item[(d)] a splitting of the weight filtration on $\Ker(b:H_V\ra H_\bbC)$;
\end{enumerate}
such that the following conditions are satisfied:
\begin{enumerate}
\item[(i)] the map $a$ induces an isomorphism $\gr^W_{-1}H_\bbC\ra \gr^W_{-1}H_V$ and the filtration on $\gr^W_{-1}H_V$ induced by the Hodge filtration on $H_V$ induces via this isomorphism a polarizable pure Hodge structure of weight $-1$ on $\gr^W_{-1}H_\bbZ$ (here $H_\bbZ$ is endowed with the filtration induced by the weight filtration on $H_\bbQ$);
\item[(ii)]  $F^0\gr^W_{0}H_V=\gr^W_{0}H_V$ and $F^0W_{-2}H_V=0$.
\end{enumerate}
Let us observe that $H_\bbZ$ underlies a canonical (polarizable) mixed Hodge structure. To see this, set
\begin{align}\label{eq:defh0-inf-add}
H^0_\add&:= \Ker(W_{-2}H_V\ra W_{-2}H_\bbC),\\
\notag
H^0_\iff&:=\Ker(\Gr^W_0H_V\ra\Gr^W_0H_\bbC).
\end{align}
The condition (i) implies $\Ker(\Gr^W_{-1}H_V\ra\Gr^W_{-1}H_\bbC)=0$
and thus the given splitting of the weight filtration on $\Ker(H_V\ra H_\bbC)$ provides a direct sum decomposition 
$H_V=H_\bbC\oplus H_\add^0\oplus H_\iff^0$ 
in which the weight filtration on $H_V$ becomes
\begin{align*}
W_0H_V&:= H_V, \\
W_{-1}H_V&:= W_{-1}H_\bbC\oplus H_\add^0,\\
W_{-2}H_V&:= W_{-2}H_\bbC\oplus H_\add^0,\\
 W_{-3}H_V&:=0. 
\end{align*}
One can then consider the one-step filtration 
\begin{equation}\label{HFH1}
0=F^{1}H_\bbC\subseteq F^0H_\bbC\subseteq F^{-1}H_\bbC=H_\bbC
\end{equation}
where $F^0H_\bbC$ is defined as the linear subspace of $H_\bbC$ formed by the elements such that there exists $v\in H_\add^0$ for which $x+v$ is contained in $\eusm F^0$. The conditions (i) and (ii) have the following consequences.

\begin{lemm}
We have $F^0H_\bbC+W_{-1}H_\bbC=H_\bbC$ and $F^0W_{-2}H_\bbC=0$. Moreover, the map $a$ induces an isomorphism $F^0\gr^W_{-1}H_\bbC=F^0\gr^W_{-1}H_V$.
\end{lemm}

In particular, $H_{\bbZ}$ underlies a canonical (polarizable) mixed Hodge structure of type $\{(0,0),(-1,0),(0,-1),(-1,-1)\}$ with Hodge filtration given by \eqref{HFH1}.

\subsection{}\label{sec:MHSM1-laumon}
Let us denote by $\MHSM_1$ the strictly full subcategory of $\MHSM$ formed by the mixed Hodge structures with modulus 
$(H,H^\bullet_\add,H^\bullet_\iff,\eusm F^\bullet)$
such that the underlying mixed Hodge structure $H$
belongs to $\MHS_1$ 
and such that $H_\iff^k=H_\add^k=0$ if $k\neq 0$.
The proof of the following proposition will be given in
\pararef{subsec:functorH1MHSM1}, \pararef{subsec:functorMHSM1H1}.

\begin{prop}\label{CompKR}
The categories $\MHSM_1$ and $\mathcal H_1$ are equivalent.
\end{prop}

\subsection{}\label{subsec:functorH1MHSM1}
Let us explain the construction of a functor from $\mathcal H_1$ to $\MHSM_1$. Let $(H_\bbZ,H_V)$ together with the data described in \pararef{subsec:KatoRussell} be an object in the category $\mathcal H_1
$. We associate with it the tuple $\eusm H=(H,H_\add^\bullet, H^\bullet_\iff,\eusm F^\bullet)$. Here $H$ is the mixed Hodge structure constructed in \pararef{subsec:KatoRussell}; $H^\bullet_\add$ and $H^\bullet_\iff$ are the sequences defined by $H^k_\add=H^k_\iff=0$ if $k\neq 0$ and
\eqref{eq:defh0-inf-add};
the Hodge filtration $\eusm F^\bullet$ is defined by $\eusm F^0:=F^0H_V$ via the canonical direct sum decomposition $H_V=H_\bbC\oplus H^0_\add\oplus H^0_\iff=:\eusm H^0$ given by the splitting of the weight filtration (see \pararef{subsec:KatoRussell}) and $\eusm F^k=F^kH_\bbC$ if $k\neq 0$. All conditions are obviously satisfied unless $k=0$. In that case, {\bf{(\ref{defi:MHSM}-a)}} is obvious and {\bf{(\ref{defi:MHSM}-b)}}  is a consequence of the definition of the Hodge filtration in \pararef{subsec:KatoRussell}. The condition  {\bf{(\ref{defi:MHSM}-c)}} is implied by $F^0\gr^W_{0}H_V=\gr^W_{0}H_V$ and {\bf{(\ref{defi:MHSM}-d)}} by $F^0W_{-2}H_V=0$.

\subsection{}\label{subsec:functorMHSM1H1}
We now construct a functor from $\MHSM_1$ to the category $\mathcal H_1$. It is easy to see that this functor and the one constructed in \pararef{subsec:functorH1MHSM1} are quasi-inverse one to another proving \propositionref{CompKR}. Given an object 
$\eusm H=(H,H_\add^\bullet, H^\bullet_\iff,\eusm F^\bullet)$ 
in $\MHSM_1$ we set 
\[H_V:=\eusm H^0=H_\bbC\oplus H_\add^0\oplus H_\iff^0\]
and 
\begin{align*}
F^{-1}H_V&:=H_V, &  W_0H_V&:= H_V, \\
F^{0}H_V&:=\eusm F^0, & W_{-1}H_V&:= W_{-1}H_\bbC\oplus H_\add^0,\\
F^{1}H_V&:=0, & W_{-2}H_V&:= W_{-2}H_\bbC\oplus H_\add^0,\\
& & W_{-3}H_V&:=0 . 
\end{align*}
The map $a:H_\bbC\ra H_V$ is given by the inclusion and the map $b:H_V\ra H_\bbC$ by the projection. The weight filtration on 
$\Ker(b)=H_\add^0 \oplus H_\inf^0$ is given by 
\[0=W_{-3}\Ker(b)\subseteq W_{-2}\Ker(b)=W_{-1}\Ker(b)=H_\add^0
\subseteq W_0\Ker(b)=\Ker(b)
\]
and the splitting is the obvious one. 
We have to check that the two conditions (i) and (ii) are satisfied. 

We start with (ii). Note that (see \eqref{eq:fund-exseq2})
$F^0W_{-2}H_V$ is a linear subspace of $\eusm F^0_\add$ and its image under the projection onto $H_\bbC$ is contained in $F^0W_{-2}H_\bbC$ which is zero (because of the restriction on the Hodge type of $H$). 
Since  the projection maps $\eusm F^0_\add$ isomorphically onto $F^0H_\bbC$ by {\bf{(\ref{defi:MHSM}-b)}} and {\bf{(\ref{defi:MHSM}-d)}}, we have $F^0W_{-2}H_V=0$. To show that $F^0\gr^W_{0}H_V=\gr^W_{0}H_V$, we have to show that $H_V=\eusm F^0+W_{-1}H_\bbC+H_\add^0$. We know that $F^0\gr^W_0H_\bbC=\gr^W_0H_\bbC$ (because of the restriction on the Hodge type of $H$), hence $H_\bbC=F^0H_\bbC+W_{-1}H_\bbC$. By {\bf{(\ref{defi:MHSM}-b)}}, we also have $F^0H_\bbC\subseteq \eusm F^0+H_\add^0$. Hence, using {\bf{(\ref{defi:MHSM}-c})}, we obtain
\[H_V=\eusm F^0+H_{\bbC}+H_\add^0\subseteq \eusm F^0+F^0H_\bbC+W_{-1}H_{\bbC}+H_\add^0\subseteq  \eusm F^0+W_{-1}H_{\bbC}+H_\add^0 \]
and therefore we have $H_V=\eusm F^0+W_{-1}H_\bbC+H_\add^0$ as desired.

To prove (i), note that (see \eqref{eq:fund-exseq2})
the restriction of $b$ to the linear subspace $W_{-1}H_V$ is induced by the projection $\eusm H^0_\add\ra H_\bbC$ which maps $\eusm F^0_\add$ isomorphically onto $F^0H_\bbC$. 
Therefore, $b$ maps $F^0W_{-1}H_V$ isomorphically onto $F^0W_{-1}H_{\bbC}$. Since $F^0W_{-2}H_V=0 $ and $F^0W_{-2}H_\bbC=0$, we have a commutative square
\[\xymatrix{{F^0W_{-1}H_V}\ar@{^(->}[r]\ar[d]^-{\simeq} & {W_{-1}H_V/W_{-2}H_V=\gr^W_{-1}H_V}\ar[d]^-{\simeq}\\
{F^0W_{-1}H_\bbC}\ar@{^(->}[r] & {W_{-1}H_\bbC/W_{-2}H_\bbC=\gr^W_{-1}H_\bbC}}\]
in which the vertical morphisms are the isomorphisms induced by $b$. This shows that the filtration on $\gr^W_{-1}H_V$ deduced from the Hodge filtration on $H_V$ is the Hodge filtration on $\gr^W_{-1}H_\bbZ$ which is thus a polarizable pure Hodge structure of weight $-1$.

\subsection{}\label{sect:Lau-MHSM1-explicit}

We denote by $\sM_1^\Lau$ the category of Laumon 1-motives over $\bbC$
 (in the sense of \cite{Laumon}).
Recall that an object of $\sM_1^\Lau$
is a two-term complex $[F \to G]$
of fppf sheaves on the category of affine $\bbC$-schemes
where 
$G$ is a connected commutative algebraic group 
and
$F$ is a formal group such that
$F = F_\et \times F_\inf$ with
$F_\et \cong \bbZ^r$ and $F_\inf \cong \hat{\bbG}_a^s$
for some $r, s \in \bbZ_{\geq 0}$.
It is proved in \cite[Theorem 4.1]{MR2985516}
that $\mathcal H_1$ is equivalent to $\sM_1^\Lau$.
Therefore \propositionref{CompKR} implies
the following corollary.

\begin{coro}\label{cor:laumon}
The categories $\MHSM_1$ and $\sM_1^\Lau$ are equivalent.
\end{coro}

Explicit description of the equivalence functors 
$\mathcal H_1 \to \sM_1^\Lau$ and $\sM_1^\Lau \to \mathcal H_1$
are given in \cite[\S 4.3 and \S 4.4]{MR2985516}.
By composing them with 
those 
in \pararef{subsec:functorH1MHSM1}
and \pararef{subsec:functorMHSM1H1},
we can explicitly describe the equivalence functors in
\corollaryref{cor:laumon} as follows.

The functor $\MHSM_1 \to \sM_1^\Lau$
sends an object 
$\eusm H=(H,H_\add^\bullet, H^\bullet_\iff,\eusm F^\bullet)$
of $\MHSM_1$ to the Laumon $1$-motive
$[F_\et \times F_\inf \to G]$ described as follows.
First, set
\begin{align*}
&G=W_{-1} H_\bbZ \backslash W_{-1}H_\bbC \oplus H_\add^0 / 
\eusm F^0 \cap (W_{-1}H_\bbC \oplus H_\add^0),
\\
& F_\et = \Gr_0^W H_\bbZ, \qquad \Lie F_\inf =H_\inf^0.
\end{align*}
Next, we describe the map $F_\et \to G$.
Consider a commutative diagram
\[
\xymatrix{
\Gr_0^W H_\bbZ \ar[r]
&
\eusm H/(W_{-1}H_\bbC \oplus H_\add^0)
&
\eusm F^0/\eusm F^0 \cap (W_{-1}H_\bbC \oplus H_\add^0)
\ar[l]_{\cong}
\\
H_\bbZ \ar@{->>}[u] \ar@{^(->}[r]
&
\eusm H  = H_\bbC \oplus H_\add^0 \oplus H_\inf^0 \ar@{->>}[u]
&
\eusm F^0. \ar@{->>}[u] \ar@{_(->}[l]
}
\]
(See \pararef{subsec:KatoRussell} (ii) for the bijectivity
of the upper right arrow.)
Given $x \in F_\et =\Gr_0^W H_\bbZ$,
we choose its lift $y \in H_\bbZ$ and
an element $z$ of $\eusm F^0$ having the same image as $x$
in $\eusm H/(W_{-1}H_\bbC \oplus H_\add^0)$.
Then $y-z \in \eusm H$ belongs to $W_{-1}H_\bbC \oplus H_\add^0$
and its class in $G$ is independent of choices of $y$ and $z$.
Therefore we get a well-defined map $F_\et \to G$.
Finally, we describe the map $F_\inf \to G$,
or what amounts to the same,
$\Lie F_\inf \to \Lie G$.
Consider a commutative diagram
\[
\xymatrix{
H_\inf^0 \ar[r]
&
\eusm H/(W_{-1}H_\bbC \oplus H_\add^0)
&
\eusm F^0/\eusm F^0 \cap (W_{-1}H_\bbC \oplus H_\add^0)
\ar[l]_{\cong}
\\
H^0_\inf \ar@{=}[u] \ar@{^(->}[r]
&
\eusm H  = H_\bbC \oplus H_\add^0 \oplus H_\inf^0 \ar@{->>}[u]
&
\eusm F^0. \ar@{->>}[u] \ar@{_(->}[l]
}
\]
Given $x \in F_\iff =H_\iff^0$,
we choose 
an element $z$ of $\eusm F^0$ having the same image as $x$
in $\eusm H/(W_{-1}H_\bbC \oplus H_\add^0)$.
Then $x-z \in \eusm H$ belongs to $W_{-1}H_\bbC \oplus H_\add^0$
and its class in 
$\Lie G=
W_{-1}H_\bbC \oplus H_\add^0 / 
\eusm F^0 \cap (W_{-1}H_\bbC \oplus H_\add^0)$
is independent of choices of $z$.
Therefore we get a well-defined map $\Lie F_\iff \to \Lie G$.

In the other direction,
the functor $\sM_1^\Lau \to \MHSM_1$
sends a Laumon $1$-motive
$[u_\et \times u_\inf : F_\et \times F_\inf \ra G]$ 
to
the object 
$\eusm H=(H,H_\add^\bullet, H^\bullet_\inf,\eusm F^\bullet)$
of $\MHSM_1$ 
described as follows.
Let $H_\bbZ$ be the fiber product of
$u_\et : F_\et \to G$ and $\exp : \Lie G \to G$,
which comes equipped with 
$\alpha : H_\bbZ \to F_\et$ and $\beta : H_\bbZ \to \Lie G$.
We set
\begin{align*}
W_0 H_\bbZ = H_\bbZ
&\supseteq
W_{-1}H_\bbZ = \ker(\alpha)=\ker(\exp)=H_1(G, \bbZ)
\\
&\supseteq W_{-2}H_\bbZ = \ker(H_1(G, \bbZ) \to H_1(G_\ab, \bbZ))
\\
&\supseteq W_{-3}H_\bbZ=0,
\end{align*}
where $G_\ab$ is the maximal Abelian quotient of $G$.
Put $H_\inf^0:=\Lie F_\inf$ and $H_\add^0 := \Lie G_\add$,
where  $G_\add$ is the additive part of $G$.
Finally, we set
\[ \eusm F^0 := \ker(H_\bbC \oplus H_\add^0 \oplus H_\inf^0
 \to \Lie G),
\]
where $H_\bbC \to \Lie G$ is induced by $\beta$,
$H_\add^0=\Lie G_\add \to \Lie G$ is the inclusion map,
and 
$H_\inf^0=\Lie F_\inf \to \Lie G$ is induced by $u_\inf$.
The Hodge filtration $F^\bullet H_\bbC$ is
determined by the condition 
\definitionref{defi:MHSM} {\bf{(\ref{defi:MHSM}-b})}.

\subsection{}\label{sect:cartier}
Let us consider
the Cartier duality functor $(\sM_1^{\Lau})^{\op} \to \sM_1^\Lau$.
The corresponding functor 
$\mathcal H_1^\op \to \mathcal H_1$
admits a simple description as
$\underline{\Hom}(-, \bbZ)(1)$ \cite[4.1]{MR2985516}.
By rewriting it through 
\pararef{subsec:functorH1MHSM1} and \pararef{subsec:functorMHSM1H1},
we find that the functor
\[ \MHSM_1^\op \to \MHSM_1,
\qquad
\eusm H \mapsto \eusm H^\vee(1)
\]
gives a duality that 
is compatible with the Cartier duality, 
via the equivalence in \corollaryref{cor:laumon}.
Here $\vee$ and $(1)$ denotes the
dual and Tate twist
(see \pararef{para:duality}, \pararef{sect:Tatetwist}).

\section{Cohomology of a variety with modulus }\label{sec:Geo}
In this section,
$X$ is a connected smooth proper variety 
of dimension $d$ over $\bbC$,
and $Y, Z$ are effective divisors on $X$
such that $|Y| \cap |Z|=\emptyset$ and 
such that $(Y+Z)_\red$ is a simple normal crossing divisor.
Put $U=X\setminus (Y\cup Z)$ and 
let us consider the following commutative diagram
\begin{equation}\label{eq:triple-diagram}
\xymatrix{
&
{U}\ar[ld]_-{j'_Y}\ar[rd]^-{j'_Z} \ar[dd]^-{j_U}
& 
\\
{X\setminus Z} \ar[rd]^-{j_Z} 
&
& 
{X\setminus Y} \ar[ld]_-{j_Y}
\\
Y \ar[r]_-{i_Y} \ar[u]^-{i_Y'} 
&
X
&
Z, \ar[l]^-{i_Z} \ar[u]_-{i_Z'} 
}
\end{equation}
where all the maps are embeddings.
The aim of this section is to construct an object
$\eusm H^n(X, Y, Z)$ of $\MHSM$
for each integer $n$.

\subsection{}\label{sect:reduced}
In this subsection we assume $Y$ and $Z$ are reduced. 
We consider the relative cohomology
\begin{equation}\label{eq:betti}
 H^n := H^n(X \setminus Z, Y, \bbZ)= H^n(X, \bbZ_{X|Y, Z})
\end{equation}
for each integer $n$, where 
$\bbZ_{X|Y, Z} := R(j_Z)_*(j'_Y)_!\bbZ_U$. 
It carries a mixed Hodge structure as in 
\cite[8.3.8]{MR0498552} or \cite[Proposition 5.46]{MR2393625}.
By applying $(Rj_Z)_*$ to an exact sequence
\[
0 \to (j_Y')_! \bbZ_U \to \bbZ_{X \setminus Z} \to (i_Y')_* \bbZ_Z \to 0,
\]
we get a canonical distinguished triangle
\begin{equation}\label{eq:rel-cpx} 
\bbZ_{X|Y, Z}\ra R(j_Z)_* \bbZ_{X \setminus Z} \ra (i_Y)_* \bbZ_Y\xra{[+1]}
\end{equation}
and therefore a long exact sequence
\begin{equation}\label{eq:loc-seq}
\dots \to 
H^n\to 
H^n(X \setminus Z, \bbZ) \to 
H^n(Y, \bbZ) \to 
H^{n+1} \to \cdots.
\end{equation}
The assumption $|Y| \cap |Z|=\emptyset$
immediately implies 
\begin{equation}\label{eq:swap-YZ}
\bbZ_{X|Y, Z} = R(j_Z)_*(j'_Y)_!\bbZ_U \cong (j_Y)_! R(j_Z')_*\bbZ_U.
\end{equation}
By applying $(j_Y)_!$ to a distinguished triangle
\[
(i_Z')_*(Ri_Z')^! \bbZ_{X \setminus Y} \to \bbZ_{X \setminus Y} \to (Rj_Z')_* \bbZ_U \to \xra{[+1]},
\]
we get a canonical distinguished triangle
\begin{equation}\label{eq:rel-cpx2} 
(i_Z)_*(Ri_Z)^! \bbZ_X \to (j_Y)_! \bbZ_{X \setminus Y} \to \bbZ_{X|Y, Z} \to \xra{[+1]}.
\end{equation}
(Note that $(i_Z)_*=(i_Z)_!$ and 
$(Ri_Z')^! \bbZ_{X \setminus Y} =(Ri_Z)^! \bbZ_X$ by excision.)
Therefore we get a long exact sequence
\begin{equation}\label{eq:loc-seq2}
\dots \to 
H^n_Z(X, \bbZ) \to 
H^n_c(X \setminus Y, \bbZ) \to 
H^n \to 
H^{n+1}_Z(X, \bbZ) \to \cdots.
\end{equation}
Both \eqref{eq:loc-seq} and \eqref{eq:loc-seq2}
are long exact sequence of $\MHS$ (see \cite{MR602463}).


We set 
\[\Omega^p_{X|Y,Z}:=\Omega_X^{p}(\log(Y+Z))\otimes \Osheaf_X(-Y)\]
where $\Omega_X^p(\log(Y+Z))$ is the sheaf on 
the analytic site of 
$X$ of $p$-forms with logarithmic poles along $(Y+Z)$. It defines a subcomplex $\Omega^\bullet_{X|Y,Z}$ of $(j_Z)_*\Omega_{X \setminus Z}^\bullet$.
For the definition of $\Omega^p_{X|Y,Z}$ in the case where $Y$ or $Z$ is non reduced we refer to \pararef{pa:nonreduced}. 

We recall the construction of the mixed Hodge structure on $H^n$, and show that its Hodge filtration can be described in terms of the complex $\Omega^\bullet_{X|Y,Z}$ as follows:
\begin{prop}\label{prop:MHS-reduced}
Suppose that $Y$ and $Z$ are reduced,
and let $n$ be an integer.
\begin{enumerate}
\item 
There is a canonical isomorphism
 \begin{equation}\label{eq:IsoOmega}
  H_\bbC^n \xra{\simeq} H^n(X, \Omega_{X|Y,Z}^{\bullet}).
  \end{equation}
 For every integer $p$, the induced map
\[
H^n(X, \Omega_{X|Y,Z}^{\bullet \ge p}) 
\to
H_\bbC^n
\]
is injective,
and its image agrees with the
Hodge filtration $F^p H_\bbC^n \subset H_\bbC^n$.
\item 
Let $p, q$ be integers.
If $\Gr_F^p \Gr^W_{p+q} H_\bbC$ is non-trivial,
then we have $p, q \in [0, n]$.
If further $n >d$, then we have $p, q \in [n-d, d]$.
\end{enumerate}
\end{prop}

Note that the isomorphism \eqref{eq:IsoOmega} is the one induced by the canonical quasi-isomorphism $\bbC_{X|Y,Z}\ra\Omega^{\bullet}_{X|Y,Z}$ (see e.g. \cite[Remarques 4.2.2 (c)]{MR894379}).

The mixed Hodge structure on $H^n$ will be constructed 
from the cohomological mixed Hodge complex
$K=(K_\bbZ, K_\bbQ, K_\bbC)$ on $X$
(in the sense of \cite[8.1.6]{MR0498552}),
and $K$ is constructed as a cone of $K^Z \to K^Y$ 
where $K^Z$ and $K^Y$ are cohomological mixed complexes
that produce the mixed Hodge structures on
$H^n(X \setminus Z, \bbZ)$ and $H^n(Y, \bbZ)$ respectively.
Therefore we first need to recall the construction of $K^Z$ and $K^Y$.

We recall from \cite[8.1.8]{MR0498552} the description of $K^Z$.
Let $K_\bbZ^Z := Rj_{Z*} \bbZ_{X \setminus Z} \in D^+(X, \bbZ)$.
Define a filtered object $(K_\bbQ^Z, W) \in D^+F(X, \bbQ)$ by
\[
K_\bbQ^Z := Rj_{Z*} \bbQ_{X \setminus Z},
\qquad
W_q K_\bbQ^Z:= \tau_{\le q} K_\bbQ^Z,
\]
where $\tau_{\le q}$ denotes the canonical truncation.
Define a bifiltered object 
$(K_\bbC^Z, W, F) \in D^+F_2(X, \bbC)$ by
\begin{align*}
&K_\bbC^Z := \Omega_X^\bullet(\log Z),
\quad
F^p K_\bbC^Z:= \Omega_X^{\bullet \ge p}(\log Z),
\quad
W_q K_\bbC^Z:= W_q \Omega_X^{\bullet}(\log Z),
\\
&
W_q \Omega_X^{p}(\log Z)=
\begin{cases}
0 & (q<0) \\
\Omega_X^{p-q}\wedge \Omega_X^q(\log Z) & (0 \le q \le p) \\
\Omega_X^p(\log Z) & (p \le q).
\end{cases}
\end{align*}
We have an obvious isomorphism 
$K_\bbZ^Z \otimes \bbQ \cong K_\bbQ^Z$,
as well as an isomorphism
$(K_\bbQ^Z, W) \otimes \bbC \cong (K_\bbC^Z, W)$
deduced from the Poincar\'e lemma.
The triple $K^Z=(K^Z_\bbZ, K^Z_\bbQ, K^Z_\bbC)$
together with these isomorphisms is a 
cohomological mixed Hodge complex
that produces the mixed Hodge structure on $H^n(X \setminus Z, \bbZ)$
(see \cite[8.1.7 and 8.1.9 (ii)]{MR0498552}).

Next we recall from \cite[3.2.4.2]{MR3290125}
the description of $K^Y$.
Let $I$ be the set of irreducible components of $Y$.
For each $k \ge 0,$ we write $Y^{[k]}$
for the disjoint union of $\cap_{T \in J} T$,
where $J$ ranges over subsets of $I$ with cardinality $k+1$.
Write $\pi^{[k]} : Y^{[k]} \to X$ for the canonical map.
Fixing an ordering on $I$, we obtain 
a complex $\bbQ_{Y^{[\bullet]}}$
and a double complex $\Omega_{Y^{[\bullet]}}^{*}$
of sheaves on $X$:
\begin{align*}
&\bbQ_{Y^{[\bullet]}} 
:=[
\pi^{[0]}_* \bbQ_{Y^{[0]}} \to \pi^{[1]}_* \bbQ_{Y^{[1]}} \to 
\cdots \to
\pi^{[q]}_* \bbQ_{Y^{[q]}} \to \cdots],
\\
&\Omega_{Y^{[\bullet]}}^{*}
:=[
\pi^{[0]}_* 
\Omega_{Y^{[0]}}^{*}
\to 
\pi^{[1]}_* 
\Omega_{Y^{[1]}}^{*}
\to 
\cdots \to
\pi^{[q]}_* 
\Omega_{Y^{[q]}}^{*} \to \cdots].
\end{align*}
Let $K_\bbZ^Y := i_{Y*} \bbZ_Y \in D^+(X, \bbZ)$.
Define a filtered object $(K_\bbQ^Y, W) \in D^+F(X, \bbQ)$ by
\[
K_\bbQ^Y := \bbQ_{Y^{[\bullet]}}
\qquad
W_q K_\bbQ^Y:= \sigma_{\ge -q} K_\bbQ^Y,
\]
where $\sigma_{\ge -q}$ denotes the brutal truncation.
Define a bifiltered object 
$(K_\bbC^Y, W, F) \in D^+F_2(X, \bbC)$ by
\begin{align*}
&K_\bbC^Y := \Tot(\Omega_{Y^{[\bullet]}}^{*}),
\quad
F^p K_\bbC^Y := 
\Tot(\sigma_{* \ge p} \Omega_{Y^{[\bullet]}}^{*}),
\quad
W_q K_\bbC^Y := 
\Tot(\sigma_{\bullet \ge -q} \Omega_{Y^{[\bullet]}}^{*}).
\end{align*}
We have a Mayer-Vietoris isomorphism 
$K_\bbZ^Y \otimes \bbQ \cong K_\bbQ^Y$,
as well as an isomorphism
$(K_\bbQ^Y, W) \otimes \bbC \cong (K_\bbC^Y, W)$
deduced from the Poincar\'e lemma.
The triple $K^Y=(K^Y_\bbZ, K^Y_\bbQ, K^Y_\bbC)$
together with these isomorphisms is a 
cohomological mixed Hodge complex
that produces the mixed Hodge structure on $H^n(Y, \bbZ)$.

We construct a morphism $K^Z \to K^Y$ of
cohomological mixed Hodge complexes.
By applying $Rj_{Z*}$ to the restriction map
$\bbZ_{X \setminus Z} \to i_{Y*}' \bbZ_Y$,
we get 
\[
K^Z_\bbZ = Rj_{Z*} \bbZ_{X \setminus Z}
\to Rj_{Z*} i_{Y*}' \bbZ_Y = i_{Y*} \bbZ_Y = K^Y_\bbZ.
\]
Similarly, we have
\begin{align*}
&K^Z_\bbQ = Rj_{Z*} \bbQ_{X \setminus Z} \to i_{Y*} \bbQ_{Y^{[0]}}
\to \pi_*^{[0]}\bbQ_{Y^{[\bullet]}} = K^Y_\bbQ,
\\
&K_\bbC^Z = \Omega_X^*(\log Z) 
\overset{\Res}{\to} \pi_*^{[0]}\Omega_{Y^{[0]}}^*
\to \Tot(\Omega_{Y^{[\bullet]}}^*) = K_\bbC^Y,
\end{align*}
where $\Res$ is the Poincar\'e residue map.
They respect filtrations
and define a morphism $\phi : K^Z \to K^Y$.
We then apply the mixed cone construction
\cite[3.22]{MR2393625}
\cite[3.3.24]{MR3290125}
to obtain $K:=\Cone(\phi)[-1]$
which produces the mixed Hodge structure on 
$H^n=H^n(X \setminus Z, Y, \bbZ)$.
\propositionref{prop:MHS-reduced} (1), (2) is a consequence
of \cite[8.1.9 (v)]{MR0498552} and \lemmaref{lem:CMHC} below,
while \propositionref{prop:MHS-reduced} (3) 
follows from \cite[8.2.4]{MR0498552} and \eqref{eq:loc-seq}.

\begin{lemm}\label{lem:CMHC}
Set 
$\Omega_{X|Y, Z}^p :=\Omega_X^p(\log (Y+Z)) \otimes \Osheaf_X(-Y)$.
Define a bifiltered object 
$(K_\bbC', W', F') \in D^+F_2(X, \bbC)$ by
\begin{align*}
&K_\bbC' := \Omega_{X|Y, Z}^\bullet,
\quad
F^p K_\bbC' :=\Omega_{X|Y, Z}^{\bullet \ge p},
\quad
W_q K_\bbC':= W_q \Omega_{X|Y, Z}^\bullet,
\\
&
W_q \Omega_{X|Y, Z}^p=
\begin{cases}
0 & (q<-p) \\
\Omega_X^{-q}\wedge \Omega_X^{p+q}(\log Y) \otimes \Osheaf_X(-Y) 
& (-p \le q < 0) \\
\Omega_X^{p-q}\wedge \Omega_X^q(\log (Y+Z))\otimes \Osheaf_X(-Y) 
& (0 \le q \le p) \\
\Omega_{X|Y, Z}^p & (p \le q).
\end{cases}
\end{align*}
Then there is a canonical map
$K_\bbC' \to K_\bbC$
inducing an isomorphism 
$(K_\bbC', W', F') \cong (K_\bbC, W, F)$
in $D^+F_2(X, \bbC)$.
\end{lemm}
\begin{proof}
The map exists since 
the composition of 
$K_\bbC' \hookrightarrow K_\bbC^Z \to K_\bbC^Y$ is the zero map.
Let us verify $K_\bbC' \to K_\bbC$ is a quasi-isomorphism.
This is a local statement,
thus it suffices to show it 
over $X \setminus Y$ and $X \setminus Z$.
The assertion becomes obvious over $X \setminus Y$,
and over $X \setminus Z$ it follows from a standard fact
\[ 
\Cone(\bbC_X \to i_{Y*} \bbC_Y)[-1]
\cong 
j_{Y!} \bbC_{X \setminus Y}
\cong 
\Omega_X^\bullet(\log Y) \otimes \Osheaf_X(-Y).
\]
A direct calculation shows that
the filtrations are transformed as described.
\end{proof}

\begin{rema}\label{rema:Saito} Let $a:X\ra \Spec(\bbC)$ be the structural morphism. By \cite{MR1047415,MR1741272} the mixed Hodge structure on $H^n$ can also be described using the six operations in the theory of mixed Hodge modules as the $n$-th cohomology group of the object $a^{\eusm H}_*(j_Z)^{\eusm H}_*(j'_Y)_!^{\eusm H}\bbQ_U^\eusm H $ of the derived category $\Db\MHS^{\bf p}_\bbQ$ of the Abelian category $\MHS^{\bf p}_\bbQ$ of polarizable mixed $\bbQ$-Hodge structures.
\end{rema}

\subsection{}\label{pa:nonreduced}
We now drop the assumption that $Y$ and $Z$ are reduced.
In this subsection, we provide an alternative description of $H_\bbC^n$.
We define
\[ \Omega_{X|Y, Z}^p 
:= \Omega_X^p(\log(Y+Z)_\red) \otimes {\Osheaf_X}(-Y+Z-Z_\red)
\subset j_{Z*} \Omega_{X \setminus Z}^p.
\]
In particular, we have (recall that $d=\dim X$)
\[
\Omega_{X|Y, Z}^0 = \Osheaf_X(-Y+Z-Z_\red)
\quad \text{and} \quad 
\Omega_{X|Y, Z}^d = \Omega_X^d \otimes \Osheaf_X(Y_\red-Y+Z).
\]
They form a subcomplex 
$\Omega_{X|Y, Z}^\bullet$ of 
$j_{Z*} \Omega_{X \setminus Z}^\bullet$.
When $Y$ and $Z$ are reduced,
this complex agrees with the one considered in the previous subsection,
and is consistent with the notation in \lemmaref{lem:CMHC}.
For another pair of effective divisors  $Y'$ and $Z'$ on $X$,
we have
\begin{equation}\label{eq:incl-omega}
\Omega_{X|Y, Z}^\bullet \subset \Omega_{X|Y', Z'}^\bullet
\quad
\text{if $Y \ge Y'$ and $Z \le Z'$}.
\end{equation}
In particular, we have
a commutative diagram of complexes
in which all arrows are inclusion maps:
\begin{equation}\label{eq:incl-qis}
\xymatrix{
& \Omega_{X|Y, Z}^\bullet \ar@{^{(}-{>}}[rd]^{} & 
\\
\Omega_{X|Y, Z_\red}^\bullet \ar@{^{(}-{>}}[ru]^{} \ar@{^{(}-{>}}[rd] &
&
\Omega_{X|Y_\red, Z}^\bullet 
\\
& \Omega_{X|Y_\red, Z_\red}^\bullet. \ar@{^{(}-{>}}[ru]^{} &
}
\end{equation}
The following proposition plays an important role in this work.
\begin{prop}\label{prop:incl-qis}
All maps in \eqref{eq:incl-qis} are quasi-isomorphisms.
Consequently, we have 
\[
H_\bbC^n 
\cong H^n(X, \Omega_{X|*, *'}^\bullet)
\]
for all $* \in \{Y, Y_\red\},~ *' \in \{Z, Z_\red \}$ and $n$.
(See \propositionref{prop:MHS-reduced} (1).)
\end{prop}
\begin{proof}
The assertion for the left lower arrow is proved in \cite[Lemma 6.1]{BS}
(under a weaker assumption that
$Y$ and $Z$ have no common irreducible component).
The same proof works for the other arrows without any change.
Though, we include a brief account here
because of its importance in our work.

By induction, it suffices to show the following:
\begin{enumerate}
\item 
For any irreducible component $T$ of $Y$,
$\Omega_{X|Y, Z}^\bullet/\Omega_{X|Y+T, Z}^\bullet$ is acyclic.
\item 
For any irreducible component $T$ of $Z$,
$\Omega_{X|Y, Z+T}^\bullet/\Omega_{X|Y, Z}^\bullet$ is acyclic.
\end{enumerate}
In what follows we outline the proof of (2)
by adopting that of \cite[Lemma 6.2]{BS} 
(which is precisely (1)).
The complex in question can be rewritten as
\[
\Omega_{X|Y, Z+T}^0 \otimes \Osheaf_T
\overset{d_T^0}{\to}
\Omega_{X|Y, Z+T}^1 \otimes \Osheaf_T
\overset{d_T^1}{\to}
\Omega_{X|Y, Z+T}^2 \otimes \Osheaf_T
\overset{d_T^2}{\to} \cdots.
\]
We have an exact sequence
\[ 
0 
\to \Omega_X^p(\log(Y+Z)_\red-T)
\to \Omega_X^p(\log(Y+Z)_\red)
\overset{\Res_T^p}{\to} \omega_T^{p-1}
\to 0,
\]
where $\omega_T^p := \Omega_T^p(\log(Z_\red-T)|_T)$
and $\Res_T^p$ is the residue map.
By taking tensor product with
$\Osheaf_X(-Y+Z-Z_\red+T) \otimes \Osheaf_T$,
we obtain another exact sequence
\begin{align*}
0 
&\to 
\Osheaf_X(-Y+Z-Z_\red+T) \otimes \omega_T^p 
\to \Omega_{X|Y, Z+T}^p \otimes \Osheaf_T
\\
&\overset{\Res_{Z, T}^p}{\to} 
\Osheaf_X(-Y+Z-Z_\red+T) \otimes \omega_T^{p-1} 
\to 0,
\end{align*}
by which we regard
$\Osheaf_X(-Y+Z-Z_\red+T) \otimes \omega_T^p$
as a subsheaf of $\Omega_{X|Y, Z+T}^p \otimes \Osheaf_T$.
Then a direct computation shows
\[ d_T^{p-1} \circ \Res_{Z, T}^p + \Res_{Z, T}^{p+1} \circ d_T^p
= e \cdot \id_{\Omega_{X|Y, Z+T}^p \otimes \Osheaf_T},
\]
where $e=- \ord_T(Z)$.
We get a homotopy operator that proves (2).
\end{proof}

\subsection{}
For any integers $k$ and $p$,
we define
\begin{align*}
\Omega_{X|Y, Z}^{(k) p}
&:=
\begin{cases}
\Omega_{X|Y, Z_\red}^p &\text{if } p < k, \\
\Omega_{X|Y_\red, Z}^p &\text{if } p \ge k
\end{cases}
\\
&=
\begin{cases}
\Omega_X^p(\log(Y+Z)_\red) \otimes \Osheaf_X(-Y) &\text{if } p < k, \\
\Omega_X^p(\log(Y+Z)_\red) \otimes \Osheaf_X(-Y_\red+Z-Z_\red) 
&\text{if } p \ge k.
\end{cases}
\end{align*}
For each $k$,
they form a subcomplex
$\Omega_{X|Y, Z}^{(k) \bullet}$ of $j_{Z*} \Omega_{X \setminus Z}^\bullet$.
We have a sequence of inclusion maps
\begin{equation}\label{eq:k-fil}
\Omega_{X|Y, Z_\red}^\bullet =
\Omega_{X|Y, Z}^{(d+1) \bullet} 
\subset 
\Omega_{X|Y, Z}^{(d) \bullet} 
\subset \dots \subset
\Omega_{X|Y, Z}^{(1) \bullet} 
\subset
\Omega_{X|Y, Z}^{(0) \bullet} 
=\Omega_{X|Y_\red, Z}^\bullet,
\end{equation}
which fits the middle row of the diagram \eqref{eq:incl-qis}.
We have
$\Omega_{X|Y_\red, Z_\red}^{(k) \bullet}
=\Omega_{X|Y_\red, Z_\red}^{\bullet}$
for any $k$.
%
%
Similarly to \eqref{eq:incl-omega},
we have
for another pair of effective divisors  $Y'$ and $Z'$ on $X$
and for any $k$
\begin{equation}\label{eq:incl-omega-k}
\Omega_{X|Y, Z}^{(k) \bullet} \subset \Omega_{X|Y', Z'}^{(k) \bullet}
\quad
\text{if $Y \ge Y'$ and $Z \le Z'$}.
\end{equation}

\begin{rema}
Some cases of the complex $\Omega_{X|Y, Z}^{(k) \bullet}$
have been used in the literature.
When $d=1$, 
$\Omega_{X|Y, Z}^{(1) \bullet}=
[\Osheaf_X(-Y) \to \Omega_X^1 \otimes \Osheaf_X(Z)]$
has been used in \cite{IY}.
When $k=d$ and $Z=\emptyset$,
$\Omega_{X|Y, \emptyset}^{(d) \bullet}$ agrees with
the complex $S_Y^\bullet$ used in \cite{MR2985516}.
\end{rema}

For integers $n$ and $k$,
we define
\begin{align}\label{eq:def-h-triples}
&\eusm H^{n, k}(X, Y, Z) 
:= H^n(X, \Omega_{X|Y, Z}^{(k) \bullet}),
\\
\notag
&\eusm F^{n, k}(X, Y, Z) 
:= H^n(X, \Omega_{X|Y, Z}^{(k) \bullet \ge k}),
\\
\notag
&(\eusm H/\eusm F)^{n, k}(X, Y, Z)
:= H^n(X, \Omega_{X|Y, Z}^{(k) \bullet <k}),
\\
\notag
&H_\add^{n, k}(X, Y)
:=  H^{n-1}\left(X, \Omega_X^{\bullet <k}(\log Y_\red) 
\otimes \frac{\Osheaf_X(-Y_\red)}{\Osheaf_X(-Y)}\right),
\\
\notag
&H_\iff^{n, k}(X, Z)
:=  H^{n-1}\left(X, \Omega_X^{\bullet <k}(\log Z_\red) 
\otimes \frac{\Osheaf_X(Z-Z_\red)}{\Osheaf_X}\right).
\end{align}
By definition we have
$\Omega_{X|Y, Z}^{(k) \bullet \ge k}
=\Omega_{X|Y_\red, Z}^{(k) \bullet \ge k},$
~
$\Omega_{X|Y, Z}^{(k) \bullet <k}
=\Omega_{X|Y, Z_\red}^{(k) \bullet <k}$
and therefore
$\eusm F^{n, k}(X, Y, Z)=\eusm F^{n, k}(X, Y_\red, Z),$
~
$(\eusm H/\eusm F)^{n, k}(X, Y, Z)=(\eusm H/\eusm F)^{n, k}(X, Y, Z_\red)$.

\begin{theo}\label{thm:coh-xyz}
Let $n$ and $k$ be integers.
\begin{enumerate}
\item 
Let $a$ and $b'$ be the 
maps induced by the inclusion maps from \eqref{eq:incl-omega-k}:
\[
\xymatrix
{
\eusm H^{n, k}(X, Y, Z_\red) \ar@/_2em/[r]^a &
\eusm H^{n, k}(X, Y, Z)  \ar@/_2em/@{.>}[l]^b \ar@/^2em/[r]_{b'} &
\eusm H^{n, k}(X, Y_\red, Z).  \ar@/^2em/@{.>}[l]_{a'}
}
\]
Then there are canonical maps $b$ and $a'$ 
such that $b \circ a = \id$ and $b' \circ a' = \id$.
\item
We have canonical isomorphisms
\begin{align*}
&\Coker(a) \cong H_\iff^{n, k}(X, Z)
\cong 
H^n\left(X, \Omega_X^{\bullet \ge k}(\log Z_\red) 
\otimes \frac{\Osheaf_X(Z-Z_\red)}{\Osheaf_X}\right),
\\
&\ker(b') \cong H_\add^{n, k}(X, Y)
\cong 
H^n\left(X, \Omega_X^{\bullet \ge k}(\log Y_\red) 
\otimes \frac{\Osheaf_X(-Y_\red)}{\Osheaf_X(-Y)}\right).
\end{align*}
\item
The sequence 
\[
0 \to \eusm F^{n, k}(X, Y, Z) \overset{i}{\longrightarrow} 
\eusm H^{n, k}(X, Y, Z) \overset{p}{\longrightarrow}
 (\eusm H/\eusm F)^{n, k}(X, Y, Z) \to 0
\]
is exact.
Hereafter we regard 
$\eusm F^{n, k}(X, Y, Z)$
as a subspace of 
$\eusm H^{n, k}(X, Y, Z)$.
\item
We have
\begin{align*}
&a(\eusm F^{n, k}(X, Y, Z_\red)) \subset \eusm F^{n, k}(X, Y, Z),
\\
&b'(\eusm F^{n, k}(X, Y, Z)) = \eusm F^{n, k}(X, Y_\red, Z).
\end{align*}
(Note however that $a'$ and $b$ do not preserve $\eusm F^{n, k}$.)
Moreover there are
commutative diagrams with exact rows and coloums
\[
\xymatrix@C=0.4cm@R=0.4cm{{} 
& 0 & 0 & & 
\\
& H_\iff^{n, k}(X, Z) \ar[u] \ar@{=}[r] & H_\iff^{n, k}(X, Z) \ar[u], & & 
\\
0 \ar[r] & 
\eusm F^{n, k}(X, Y, Z) \ar[u] \ar[r] &
\eusm H^{n, k}(X, Y, Z) \ar[u] \ar[r] &
(\eusm H/\eusm F)^{n, k}(X, Y, Z) \ar[r] & 0
\\
0 \ar[r] & 
\eusm F^{n, k}(X, Y, Z_\red) \ar[u] \ar[r] &
\eusm H^{n, k}(X, Y, Z_\red) \ar[u]_a \ar[r] &
(\eusm H/\eusm F)^{n, k}(X, Y, Z_\red) \ar[r] \ar@{=}[u] & 0
\\
& 0 \ar[u] & 0, \ar[u] & & 
}
\]
and
\[
\xymatrix@C=0.4cm@R=0.4cm{{} 
& & 0 \ar[d] & 0 \ar[d] & 
\\
& & H_\add^{n, k}(X, Y) \ar[d] \ar@{=}[r] & H_\add^{n, k}(X, Y) \ar[d] & 
\\
0 \ar[r] & 
\eusm F^{n, k}(X, Y, Z) \ar@{=}[d] \ar[r] &
\eusm H^{n, k}(X, Y, Z) \ar[d]^{b'} \ar[r] &
(\eusm H/\eusm F)^{n, k}(X, Y, Z) \ar[r] \ar[d] & 0
\\
0 \ar[r] & 
\eusm F^{n, k}(X, Y_\red, Z) \ar[r] &
\eusm H^{n, k}(X, Y_\red, Z) \ar[d] \ar[r] &
(\eusm H/\eusm F)^{n, k}(X, Y_\red, Z) \ar[r] \ar[d] & 0
\\
& & 0 & 0. & 
}
\]
\item
The inclusion map from \eqref{eq:k-fil} induces a map
\[
\tau^{n, k} : 
\eusm H^{n, k}(X, Y, Z) \to \eusm H^{n, k-1}(X, Y, Z)
\]
and it holds that
$\tau^{n, k}(\eusm F^{n, k}(X, Y, Z))
\subset \eusm F^{n, k-1}(X, Y, Z).$
The same map also induces maps
\begin{align*}
&\tau^{n, k}_\add : 
H_\add^{n, k}(X, Y) \to H_\add^{n, k-1}(X, Y),
\\
&\tau^{n, k}_\iff : 
H_\iff^{n, k}(X, Z) \to H_\iff^{n, k-1}(X, Z).
\end{align*}
\end{enumerate}
\end{theo}
\begin{proof}
We introduce complexes 
$\Omega_{X|Y, Z}^{(k)' \bullet}$ and
$\Omega_{X|Y, Z}^{(k)'' \bullet}$ by setting
\begin{align*}
\Omega_{X|Y, Z}^{(k)' p}
:=
\begin{cases}
\Omega_{X|Y, Z_\red}^p &\text{if } p < k, \\
\Omega_{X|Y, Z}^p &\text{if } p \ge k,
\end{cases}
\quad
\Omega_{X|Y, Z}^{(k)'' p}
:=
\begin{cases}
\Omega_{X|Y, Z}^p &\text{if } p < k, \\
\Omega_{X|Y_\red, Z}^p &\text{if } p \ge k.
\end{cases}
\end{align*}
Altogether, they fit in a diagram extending \eqref{eq:incl-qis}
in which all arrows are inclusions:
\begin{equation}\label{eq:incl-big}
\xymatrix@R=.50cm@C=.50cm{
& & \Omega_{X|Y, Z}^\bullet \ar@{^{(}-{>}}[rd]^{} & &
\\
& \Omega_{X|Y, Z}^{(k)' \bullet} \ar@{^{(}-{>}}[rd]^{} \ar@{^{(}-{>}}[ru]^{} & 
& \Omega_{X|Y, Z}^{(k)'' \bullet} \ar@{^{(}-{>}}[rd]^{} &
\\
\Omega_{X|Y, Z_\red}^\bullet \ar@{^{(}-{>}}[ru]^{} \ar@{^{(}-{>}}[rd] & 
& \Omega_{X|Y, Z}^{(k) \bullet} \ar@{^{(}-{>}}[rd]^{} 
\ar@{^{(}-{>}}[ru]^{\beta} & &
\Omega_{X|Y_\red, Z}^\bullet 
\\
& \Omega_{X|Y, Z_\red}^{(k) \bullet} \ar@{^{(}-{>}}[rd]^{} 
\ar@{^{(}-{>}}[ru]^{\alpha} & 
& \Omega_{X|Y_\red, Z}^{(k) \bullet} \ar@{^{(}-{>}}[ru]^{} &
\\
& & \Omega_{X|Y_\red, Z_\red}^\bullet. \ar@{^{(}-{>}}[ru]^{} & &
}
\end{equation}
The map $a$ is induced by $\alpha$.
The cokernel of the degree $p$ part of $\beta \circ \alpha$
is given by
\begin{align*} 
&\Omega^p_X(\log(Y+Z)_\red) \otimes
\frac{\Osheaf_X(-Y+Z-Z_\red)}{\Osheaf_X(-Y)}
\quad \text{for} \quad p <k,
\\
&\Omega^p_X(\log(Y+Z)_\red) \otimes
\frac{\Osheaf_X(-Y_\red+Z-Z_\red)}{\Osheaf_X(-Y_\red)}
\quad \text{for} \quad p \ge k,
\end{align*}
but since $|Y| \cap |Z|=\emptyset$ we have
\[\frac{\Osheaf_X(-Y+Z-Z_\red)}{\Osheaf_X(-Y)}
\cong 
\frac{\Osheaf_X(-Y_\red+Z-Z_\red)}{\Osheaf_X(-Y_\red)}
\cong
\frac{\Osheaf_X(Z-Z_\red)}{\Osheaf_X}.
\]
It follows that the cokernel of $\beta \circ \alpha$ is isomorphic to
that of the lower right arrow in \eqref{eq:incl-qis},
hence acyclic by \propositionref{prop:incl-qis}.
We have shown that $\beta \circ \alpha$ is a quasi-isomorphism.
We define $b$ to be the composition of 
$H^n(\beta)$ and $H^n(\beta \circ \alpha)^{-1}$,
showing the first half of (1).
Since we have seen that $a$ is injective for any $n$,
we also get
$\Coker(a) \cong H^n(X, \Coker(\alpha))$,
which is nothing but 
the right hand side of the first displayed formula in (2).
Similarly, as we have seen $b$ is surjective for any $n$,
we get $\Coker(a) \cong \ker(b) \cong 
H^{n-1}(X, \Coker(\beta)) = H_\iff^{n, k}(X, Z)$,
proving the first half of (2).
The rest of (1) and (2) is shown by a dual argument.

To prove (3), we consider a diagram
\[
\xymatrix{
& \Omega_{X|Y_\red, Z}^{(k) \bullet \ge k} \ar[r]^{f}
& \Omega_{X|Y_\red, Z}^{(k) \bullet} \ar[r]^{g}
& \Omega_{X|Y_\red, Z}^{(k) \bullet < k } &
\\
\Omega_{X|Y, Z}^{(k) \bullet \ge k} \ar[r]^{f'}
\ar@{=}[ur]
& \Omega_{X|Y, Z}^{(k) \bullet} \ar@{^{(}->}[ur] &
& \Omega_{X|Y_\red, Z_\red}^{(k) \bullet} \ar[r]^{g'}
\ar@{_{(}->}[ul]
& \Omega_{X|Y_\red, Z_\red}^{(k) \bullet <k}. \ar@{=}[ul]
}
\]
By \propositionref{prop:MHS-reduced} (2),
$g'$ induces surjections on cohomologies,
hence so does $g$.
It follows that $f$ induces injections on cohomologies,
hence so does $f'$.
We have shown the injectivity of $i$.
A dual argument proves the surjectivity of $p$.
(3) follows.

The first half of (4) is obvious.
The rest is obtained by taking cohomology of the diagram:
\[
\xymatrix{
& 0 & 0 & &
\\
&
\frac{\Omega_{X|Y, Z}^{(k) \bullet \ge k}}{\Omega_{X|Y, Z_\red}^{(k) \bullet \ge k}} \ar@{=}[r] \ar[u]
& 
\frac{\Omega_{X|Y, Z}^{(k) \bullet}}{\Omega_{X|Y, Z_\red}^{(k) \bullet}} \ar[u] & 
&
\\
0 \ar[r] &
\Omega_{X|Y, Z}^{(k) \bullet \ge k} \ar[r] \ar[u]
& 
\Omega_{X|Y, Z}^{(k) \bullet} \ar[r] \ar[u]
& 
\Omega_{X|Y, Z}^{(k) \bullet < k } \ar[r]
& 0
\\
0 \ar[r] &
\Omega_{X|Y, Z_\red}^{(k) \bullet \ge k} \ar[r] \ar[u]
& 
\Omega_{X|Y, Z_\red}^{(k) \bullet} \ar[r] \ar[u]
& 
\Omega_{X|Y, Z_\red}^{(k) \bullet < k } \ar@{=}[u] \ar[r]
& 0
\\& 0 \ar[u] & 0 \ar[u] & &
}
\]
and its dual diagram.
(5) is obvious.
\end{proof}

\begin{coro}\label{cor:vanish-addinf}
We have $H_\add^{n,k}(X, Y)=H_\iff^{n,k}(X, Z)=0$
if one of the following conditions is satisfied:
\begin{enumerate}
\item $k \le 0$ or $k<n-d+1$.
\item $k>d$ or $k>n$.
\end{enumerate}
\end{coro}
\begin{proof}
Since the sheaves
$\Osheaf_X(-Y_\red)/\Osheaf_X(-Y)$ and 
$\Osheaf_X(Z-Z_\red)/\Osheaf_X$ are 
supported in a closed subvariety of dimension $d-1$,
(1) follows from the definition \eqref{eq:def-h-triples}.
\theoremref{thm:coh-xyz} (2) implies the case (2).
\end{proof}

\subsection{}
We arrive at our main definition.
\begin{defi}\label{def:MHSMofMTri}
For each integer $n$,
we define an object
\[ \eusm H^n(X, Y, Z)=
(H^n, H_\add^{n, \bullet}, H_\iff^{n, \bullet}, \eusm F^{n, \bullet})
\]
of $\MHSM$ as follows.
Let $H^n$
be the mixed Hodge structure considered in \S \ref{sect:reduced}.
We define two objects 
$H_\iff^{n, \bullet}$
and $H_\add^{n, \bullet}$ of $\Vec^\bullet_\bbC$
to be 
$(H_\iff^{n, k}(X, Z), \tau_\iff^{n, k})_k$
and 
$(H_\add^{n, k}(X, Y), \tau_\add^{n, k})_k$
respectively.
For each $k \in \bbZ$
we have 
\begin{align*} 
\eusm H^{n, k}(X, Y, Z) 
&\cong 
\eusm H^{n, k}(X, Y_\red, Z) \oplus H_\add^{n, k}(X, Y)
\\
&\cong 
\eusm H^{n, k}(X, Y, Z_\red) \oplus H_\inf^{n, k}(X, Z)
\\
&\cong 
\eusm H^{n, k}(X, Y_\red, Z_\red) \oplus 
H_\add^{n, k}(X, Y) \oplus H_\iff^{n, k}(X, Z).
\\
&\cong 
H_\bbC^n \oplus 
H_\add^{n, k}(X, Y) \oplus H_\iff^{n, k}(X, Z).
\end{align*}
Here we applied \theoremref{thm:coh-xyz} (1--2)
(to $(X, Y, Z), (X, Y_\red, Z)$ and $(X, Y, Z_\red)$),
and for the last isomorphism
we used \propositionref{prop:MHS-reduced} (1).
We then define
$\eusm F^{n, k} \subset H_\bbC^n \oplus H_\iff^{n,k} \oplus H_\add^{n,k}$
to be the subspace
corresponding to 
$\eusm F^{n, k}(X, Y, Z) \subset \eusm H^{n, k}(X, Y, Z)$,
using \theoremref{thm:coh-xyz} (3).
\theoremref{thm:coh-xyz} (4--5) implies that
they satisfy the conditions 
{\bf{(\ref{defi:MHSM}-a})}-{\bf{(\ref{defi:MHSM}-d})}
in \definitionref{defi:MHSM}.
\end{defi}

Note that the mixed Hodge structures with modulus introduced in \definitionref{def:MHSMofMTri} are polarizable (that is, their underlying mixed Hodge structures are).


\subsection{}\label{sect:funct}
Let $(X, Y, Z)$ and $(X', Y', Z')$ be two triples
as in the beginning of \pararef{sec:Geo}.
Let $f : X \to X'$ be a morphism of $\bbC$-schemes
such that $f(X) \not\subset |Y'| \cup |Z'|$.
If $f$ verifies the conditions
\begin{equation}\label{eq:functorial} 
Y \le f^* Y', \quad Z_\red \ge (f^* Z')_\red, \quad
Z-Z_\red \ge f^*(Z'-Z_\red'),
\end{equation}
then it induces a morphism
$f^* : \eusm H^n(X', Y', Z') \to \eusm H^n(X, Y, Z)$
for any $n$.
To see this, 
we first note that the first two items in \eqref{eq:functorial}
implies that $f$ restricts to 
$X \setminus Z_\red \to X' \setminus Z_\red'$
and to $Y_\red \to Y_\red'$.
Hence we have a pull-back map
$f^* : 
H^n(X' \setminus Z_\red', Y_\red', \bbZ) \to H^n(X \setminus Z_\red, Y_\red, \bbZ)$
in $\MHS$.
We next note that 
\eqref{eq:functorial} implies
$\Omega_{X|f^*Y', f^*Z'}^{(k) \bullet} \subset \Omega_{X|Y, Z}^{(k) \bullet}$
for any $k \in \bbZ$.
Hence  we have a pull-back map
$f^* : \eusm H^{n, k}(X', Y', Z') \to \eusm H^{n, k}(X, Y, Z)$
induced by the maps of complexes
$\Omega_{X'|Y', Z'}^{(k) \bullet} 
\to f_* \Omega_{X|f^*Y', f^*Z'}^{(k) \bullet} 
\to f_* \Omega_{X|Y, Z}^{(k) \bullet}$,
which verifies $f^*(\eusm F^{n, k}(X',Y',Z')) \subset \eusm F^{n, k}(X, Y, Z)$.
By \theoremref{thm:coh-xyz} (2)
it induces
$f^* : H_\add^{n, k}(X', Y') \to H_\add^{n, k}(X, Y)$
and
$f^* : H_\inf^{n, k}(X', Z') \to H_\inf^{n, k}(X, Z)$.
These maps define the desired morphism.

\begin{rema}
Composition of two morphisms satisfying \eqref{eq:functorial}
need not satisfy \eqref{eq:functorial}.
Here is an example:
$(\Spec \bbC, \emptyset, \emptyset)
\to (\mathbb{P}^1, \emptyset, \emptyset)
\to (\mathbb{P}^1, x, \emptyset)$,
where the first map is the immersion to 
a closed point $x$,
and the second map is given by the identity.
\end{rema}

\subsection{}\label{sect:curves}
As an example, 
we give an explicit description of
$\eusm H =\eusm H^1(X, Y, Z)$
when $d=1$.
Write $\eusm H=(H, H_\inf^\bullet, H_\add^\bullet, \eusm F^\bullet)$
so that $H=H^1(X \setminus Z, Y, \bbZ)$.
If $k \not= 1$, then we have $H_\inf^k=H_\add^k=0$ 
and hence 
$\eusm H^k = H_\bbC, ~\eusm F^k = F^k H_\bbC$.
We have
\begin{align*}
&H_\bbC \cong H^1(X, [\Osheaf_X(-Y_\red) \to \Omega_X^1 \otimes \Osheaf_X(Z_\red)]),
\\
&F^1 H_\bbC 
\cong H^0(X, \Omega_X^1 \otimes \Osheaf_X(Z_\red)),
\quad F^2 H_\bbC=H_\bbC, \quad F^0 H_\bbC=0,
\\
&H_\add^1=H^0(X, \Osheaf_X(-Y_\red)/\Osheaf_X(-Y))
\cong H^0(X, \Omega_X^1 \otimes (\Osheaf_X/\Osheaf_X(Y_\red-Y))),
\\
&H_\iff^1=H^0(X, \Osheaf_X(Z-Z_\red)/\Osheaf_X)
\cong H^0(X, \Omega_X^1 \otimes (\Osheaf_X(Z)/\Osheaf_X(Z_\red))),
\\
&\eusm H^{1}=H^1(X, [\Osheaf_X(Y) \to \Omega_X^1 \otimes \Osheaf_X(Z)])
\cong H_\bbC \oplus H_\add^1 \oplus H_\inf^1,
\\
&\eusm F^{1}=H^0(X, \Omega_X^1 \otimes \Osheaf_X(Z)).
\end{align*}
These gadgets are considered in \cite[Propositions 10, 14 and Definition 13]{IY}.

\section{Duality}\label{sec:dual}

Throughout this section, 
let $X$ be a connected smooth proper variety 
of dimension $d$ over $\bbC$,
and let $Y, Z$ be effective divisors on $X$
such that $|Y| \cap |Z|=\emptyset$ and 
such that $(Y+Z)_\red$ is a simple normal crossing divisor. 
The main result of this section is the following.

\begin{theo}\label{prop:dualitytriple}
For every integer $n\in\bbZ$, there exists an isomorphism in $\MHSM$
\[\eusm H^n(X,Y,Z)^\vee \cong \eusm H^{2d-n}(X,Z,Y)(d)_\fr, \]
where $(-)^\vee$ is the duality functor described in \pararef{para:duality},
and 
$(-)(d)$ 
(resp. $(-)_\fr$)
is the Tate twist (resp. free part)
introduced in \pararef{sect:Tatetwist} (resp. \pararef{sect:free}).
\end{theo}

\subsection{} 
Let us first assume that $Y,Z$ are reduced. 
For a $\bbC$-scheme $V$ with structural map $a : V \to \Spec A$, 
let $\Dbc(V,A)$ the bounded derived category of sheaves of 
$A$-modules with (algebraically) constructible cohomology and
 $\bbD_V:=R\Homo(-, a^! A) : \Dbc(V,A) \to \Dbc(V,A)$ 
be the Verdier duality functor,
where $A$ is $\bbZ, \bbQ$ or $\bbC$.
Since we have $\bbD_U(A_U)=A_U(d)[2d]$
as $U=X \setminus (|Y| \cup |Z|)$ is smooth of dimension $d$,
we have
\begin{align}
\label{eq:duality-derived}
\bbD_X(A_{X|Y, Z})&=\bbD_X(Rj_{Z*} j_{Y!}' A_U)
=j_{Z!} Rj_{Y*}' \bbD_U(A_U)
\\
\notag
&\overset{(*)}{=}Rj_{Y*} j_{Z!}' A_U(d)[2d]
=A_{X|Z, Y}(d)[2d],
\end{align}
where we used the notations from
\eqref{eq:triple-diagram} for the maps $j_Y$, etc.
(see \eqref{eq:swap-YZ} for $(*)$).
The induced pairing 
$A_{X|Y, Z} \otimes A_{X|Z, Y} \to A_X$
factors through
\begin{equation}\label{eq:cupproduct0}
A_{X|Y, Z} \otimes A_{X|Z, Y} \to (j_U)_! A_U,
\end{equation}
since 
$(A_{X|Y, Z})_y=0$ for all $y \in Y$
and
$(A_{X|Z, X})_z=0$ for all $z \in Z$
(as extension by zero).
Therefore we obtain a pairing 
\begin{equation}\label{eq:cupproduct}
H^n(X\setminus Z,Y, A)\otimes H^{2d-n}(X\setminus Y,Z,A)(d)\xra{\smile}H_c^{2d}(U,A)(d)
\cong A,
\end{equation}
where the last isomorphism is the trace map.
As this pairing is perfect, up to torsion when $A=\bbZ$, 
we obtain an isomorphism
\begin{equation}\label{eq:duality1}
 \Hom_A(H^n(X\setminus Z,Y, A), A) \cong H^{2d-n}(X\setminus Y,Z, A)(d)_\fr. 
\end{equation}
(Here $(-)_\fr$ makes no effect if $A=\bbQ$ or $\bbC$.)

We have a canonical pairing 
\begin{equation}\label{eq:cupprod-dR}
\Omega^{\bullet}_{X|Y,Z}\otimes_\bbC\Omega^{\bullet}_{X|Z,Y}\ra 
{}^c\Omega_U^\bullet 
:=\Omega^{\bullet}_X(\log(Y+Z)_\red) \otimes \Osheaf_X(-Y_\red-Z_\red)
\end{equation}
defined by the wedge product.
Since
${}^c\Omega_U^\bullet$ is a resolution of $(j_U)_! \bbC_U$,
we obtain a pairing 
\begin{equation}\label{eq:cupproduct2}
H^n(X, \Omega^\bullet_{X|Y, Z})
\otimes H^{2d-n}(X, \Omega^\bullet_{X|Z, Y})
\xra{\smile} 
H^{2d}_c(U, \bbC) \cong \bbC,
\end{equation}
where the last isomorphism is the trace map.

The two pairings \eqref{eq:cupproduct} (with $A=\bbC$)
and \eqref{eq:cupproduct2} 
are compatible with respect to isomorphisms
$H^n(X \setminus Z, Y, \bbC) \cong H^n(X, \Omega^\bullet_{X|Y, Z})$
and
$H^{2d-n}(X \setminus Y, Z, \bbC) \cong H^{2d-n}(X, \Omega^\bullet_{X|Z, Y})$
from \propositionref{prop:MHS-reduced}.
This follows from a commutative diagram
\[
\xymatrix{
\bbC_{X|Y, Z} \otimes \bbC_{X|Z, Y} \ar[r] \ar[d]
&
(j_U)_!\bbC_U \ar[d]
\\
\Omega^\bullet_{X|Y, Z} \otimes \Omega^\bullet_{X|Z, Y} \ar[r]
&
{}^c\Omega_U^\bullet,
}
\]
where 
the horizontal maps are \eqref{eq:cupproduct0} and \eqref{eq:cupprod-dR}.

\begin{lemm}\label{lemm:MHS}
The morphism \eqref{eq:duality1} defines an isomorphism of mixed Hodge structures
\[H^{n}(X,Y,Z)^\vee \xra{\simeq} H^{2d-n}(X,Z,Y)(d)_\fr.\]
\end{lemm}
\begin{proof}
Since we already know it is an isomorphism
of underlying Abelian groups,
it suffices to prove it is a morphism of $\MHS$.
(This is known 
when $Y=\emptyset$ or $Z=\emptyset$, see \cite{MR602463}.)
This amounts to showing that under the pairing 
\eqref{eq:cupproduct}
\[
F^p H^n(X\setminus Z,Y, \bbC)
~\text{and}~ 
F^{1-p}(H^{2d-n}(X\setminus Z,Y, \bbC)(d))=F^{d+1-p}H^{2d-n}(X\setminus Z,Y, \bbC)
\]
annihilate each other,
and the same for
\[
W_q H^n(X\setminus Z,Y, \bbQ)
~\text{and}~ 
W_{-q-1}(H^{2d-n}(X\setminus Z,Y, \bbQ)(d))=W_{2d-q-1}H^{2d-n}(X\setminus Z,Y, \bbQ).
\]
Both immediately follow from the description of the filtrations 
given in \lemmaref{lem:CMHC}.
\end{proof}

\begin{rema}\label{rema:cupsixop}
Alternatibly, one can prove this lemma in the same way as
\eqref{eq:duality-derived}
upon replacing $\Dbc(V,A)$
by the category of mixed Hodge modules $\MHM(V,\bbQ)$ \cite{MR1047415}.
\end{rema}

\subsection{}
Let $C^\bullet,D^\bullet$ be complexes of sheaves of $\bbC$-vector spaces on $X$. 
Let $\wedge:C^\bullet\otimes D^\bullet\ra {}^c\Omega_U^\bullet$ be a morphism of complexes. 
The map $\wedge$ together with the trace map induces a morphism
\[\bH^i(X,C^\bullet)\otimes_\bbC \bH^{2d-i}(X,D^\bullet)(d)\xra{\wedge}
\bH^{2d}(X,{}^c\Omega_U^\bullet)(d)\xra{\simeq}H^{2d}_c(U,\bbC)(d)\xra{\Tr}\bbC\] 
that defined a canonical morphism
\begin{equation}\label{eq:dualitySerre}
\bH^{2d-i}(X,D^\bullet)(d)\ra \bH^i(X,C^\bullet)^\vee.
\end{equation}
Note that ${}^c\Omega_U^d=\Omega_X^d$ by definition.
\begin{lemm}\label{lemm:Serre}
Assume that $C^i$ and $D^i$ are locally free $\Osheaf_X$-modules for all $i$
and that $C^i=D^i=0$ if $i \not\in [0, d]$. If the map $\wedge$ induces an isomorphism
$C^i \xra{\simeq}\Homo_{\Osheaf_X}(D^{d-i}, \Omega^d_X)$ for all $i$, then the morphism \eqref{eq:dualitySerre} is an isomorphism for every integer $i$. 
\end{lemm}
\begin{proof}
Let $n\in[0,d]$ be the greatest integer such that $C^n\neq 0$. Consider the truncated complexes $C'^\bullet=C^{\bullet < n}$ and 
$D'^\bullet=D^{\bullet>d-n}$ so that we have  the exact sequences of complexes of sheaves of $\bbC$-vector spaces on $X$
\[
0 \to C^n[-n] \to C^\bullet \to C'^\bullet \to 0,
\quad
0 \to D'^\bullet \to D^\bullet \to D^{d-n}[n-d] \to 0.
\]
These sequences, together with the pairing $\wedge$ and the trace map, induce a morphism of long exact sequences
\[\xymatrix@C=.5cm{{\cdots}\ar[r] &{\bH^{2d-i}(X,D'^\bullet)(d)}\ar[r]\ar[d] & {\bH^{2d-i}(X,D^\bullet)(d)}\ar[r]\ar[d] & {\bH^{2d-i}(X,C^n[-n])(d)}\ar[r]\ar[d]  &{\cdots}\\
{\cdots}\ar[r] & {\bH^i(X,C'^\bullet)^\vee}\ar[r] & {\bH^i(X,C^\bullet)^\vee}\ar[r] & {\bH^i(X,D^{d-n}[n-d])^\vee}\ar[r]& {\cdots.}}\]
Let $\scr F=D^{d-n}$ and $r=2d-i-n$. Note that 
\[\bH^{2d-i}(X,C^n[-n])=\bH^{2d-i-n}(X,\Homo_{\Osheaf_X}(\scr F,\Omega^d_X))=\Ext^{r}_{\Osheaf_X}(\scr F,\Omega^d_X)\]
and $\bH^i(X,D^{d-n}[n-d])=\bH^{d-r}(X,\scr F)$.
Hence, we are reduced by induction to showing that the map
\[\Ext^r_{\Osheaf_X}(\scr F,\Omega^d_X)\ra H^{d-r}(X,\scr F)^\vee\]
induced by the canonical pairing
\[\Ext^r_{\Osheaf_X}(\scr F,\Omega^d_X)\times H^{d-r}(X,\scr F)\ra H^{d}(X,\Omega^d_X)\xra{\Tr}\bbC\]
is an isomorphism. This is Serre duality from \cite{MR0068874} (see also \cite[Theorem 7.6]{MR0463157}).
\end{proof}

\subsection{}
Let us come back to the general assumption of \theoremref{prop:dualitytriple} (that is $Y$ and $Z$ may not be reduced).  Let $k\in\bbZ$ be an integer. Then we have a canonical pairing 
\[\wedge:\Omega_{X|Y, Z}^{(k)\bullet}\otimes_\bbC\Omega_{X|Z, Y}^{(d+1-k)\bullet}\ra
{}^c\Omega^\bullet_U\]
that induces a morphism (for every $n\in\bbZ$)
\[\eusm H^{n,k}(X,Y,Z)\otimes_\bbC \eusm H^{2d-n,d+1-k}(X,Z,Y)\ra H^{2d}(X,{}^c\Omega^\bullet_U)
\cong \bbC
\]
and therefore via the trace map a morphism 
\begin{equation}\label{eq:MorDua}
\eusm H^{2d-n,d+1-k}(X,Z,Y)\ra \eusm H^{n,k}(X,Y,Z)^\vee
\end{equation}
By \lemmaref{lemm:Serre}, applied with  $C^\bullet=\Omega_{X|Y, Z}^{(k)\bullet}$ and $D^\bullet=\Omega_{X|Z, Y}^{(d+1-k)\bullet}$, we see that the morphism \eqref{eq:MorDua} is an isomorphism for every integer $n,k\in\bbZ$. 
\theoremref{thm:coh-xyz} (2) also implies that \eqref{eq:MorDua} induces isomorphisms
\begin{align}
&H_\add^{n, k}(X, Y)^\vee \cong H_\inf^{2d-n, d+1-k}(X, Y),\label{eq:isoaddinf}
\\
&H_\inf^{n, k}(X, Z)^\vee \cong H_\add^{2d-n, d+1-k}(X, Z)\notag.
\end{align}
To prove \theoremref{prop:dualitytriple} it suffices to show the following lemma.
\begin{lemm}
Under the duality of \eqref{eq:MorDua},
$\eusm F^{n, k}(X, Y, Z)$ is
the exact annihilator of
$\eusm F^{2d-n, d+1-k}(X, Z, Y)$. 
\end{lemm}
\begin{proof}
They annihilate each other for the simple reason of degrees.
Thus it suffices to show
\[ \dim \eusm H^{n, k}(X, Y, Z)/\eusm F^{n, k}(X, Y, Z)
=
   \dim \eusm F^{2d-n, d+1-k}(X, Z, Y).
\]
By \eqref{eq:fund-exseq1} and \eqref{eq:fund-exseq3},
the left hand side is equal to 
\[ \dim H^{n}(X \setminus Z, Y, \bbC)/F^k H^{n}(X \setminus Z, Y, \bbC)
+ \dim H_\add^{n, k}(X, Y),
\]
while 
by \eqref{eq:fund-exseq2} and \eqref{eq:fund-exseq3},
the right hand side is equal to
\[ \dim F^{d+1-k} H^{2d-n}(X \setminus Y, Z, \bbC)
+ \dim H_\inf^{2d-n, d+1-k}(X, Y).
\]
The first terms coincide by \lemmaref{lemm:MHS},
and the second terms also agree
since we have the isomorphisms \eqref{eq:isoaddinf}.
\end{proof}

This completes the proof of \theoremref{prop:dualitytriple}.

\subsection{}\label{sect:push}
Let $(X', Y', Z')$ be another triple
as in the beginning of this section
and set $d'=\dim X'$.
Let $f : X \to X'$ be a morphism of $\bbC$-schemes
such that $f(X) \not\subset |Y'| \cup |Z'|$.
If $f$ verifies the conditions
\begin{equation}\label{eq:functorial2} 
Y_\red \ge (f^* Y')_\red, \quad,
Y-Y_\red \ge f^*(Y'-Y_\red'), \quad
Z \le f^*(Z'),
\end{equation}
then 
by \theoremref{prop:dualitytriple} and \pararef{sect:funct}
we obtain an induced map
\[ f_* : \eusm H^{2d-n}(X, Y, Z)(d)_\fr 
\to \eusm H^{2d'-n}(X', Y', Z')(d')_\fr, 
\]
which id dual to 
$f^* :\eusm H^{n}(X', Z', Y') \to \eusm H^{n}(X, Z, Y).$

\section{Picard and Albanese $1$-motives}\label{sect:pic-alb}
\subsection{}
Let $(X, Y, Z)$ be as in the beginning of \pararef{sec:Geo},
and let us consider 
the objects 
$\eusm H^1(X, Y, Z)$ and $\eusm H^{2d-1}(X, Y, Z)$ 
in $\MHSM$ from \definitionref{def:MHSMofMTri}.
\propositionref{prop:MHS-reduced} (3)
and \corollaryref{cor:vanish-addinf}
show that
$\eusm H^1(X, Y, Z)(1)_\fr$ and $\eusm H^{2d-1}(X, Y, Z)(d)_\fr$ 
belong to $\MHSM_1$ (see \pararef{sec:MHSM1-laumon}),
where $(m)$ denotes the Tate twists 
and $(-)_\fr$ denotes the free part
(see \pararef{sect:Tatetwist}, \pararef{sect:free}).

\begin{defi}
We define 
the Picard and Albanese $1$-motives 
$\Pic(X, Y, Z)$ and $\Alb(X, Y, Z)$
to be the Laumon $1$-motive 
that corresponds to 
$\eusm H^1(X, Y, Z)(1)_\fr$ and $\eusm H^{2d-1}(X, Y, Z)(d)_\fr$,
respectively,
under the equivalence $\MHSM_1 \cong \sM_1^\Lau$
from \corollaryref{cor:laumon}.
\end{defi}

In view of \pararef{sect:cartier} and \eqref{eq:free-dual},
\theoremref{prop:dualitytriple}
shows that
$\Pic(X, Y, Z)$ and $\Alb(X, Z, Y)$
are dual to each other under the Cartier duality.

\subsection{}
Let $(X, Y, Z)$ and $(X', Y', Z')$ be two triples
as in the beginning of \pararef{sec:Geo}
and put $d=\dim X$ and $d'=\dim X'$.
Let $f : X \to X'$ be a morphism of $\bbC$-schemes
such that $f(X) \not\subset |Y'| \cup |Z'|$.
If the conditions \eqref{eq:functorial} are satisfied,
then by \pararef{sect:funct} 
there is an induced map
$f^* : \eusm H^1(X', Y', Z')(1) \to \eusm H^1(X, Y, Z)(1)$
and hence we obtain
\[ f^* : \Pic(X', Y', Z') \to \Pic(X, Y, Z). \]
Similarly, if
the conditions \eqref{eq:functorial2} are satisfied,
then by \pararef{sect:push},
there is an induced map
$f_* : \eusm H^{2d-1}(X, Y, Z)(d)_\fr \to \eusm H^{2d'-1}(X', Y', Z')(d')_\fr$
and hence we obtain
\[  f_* : \Alb(X, Y, Z) \to \Alb(X', Y', Z').
\]

\subsection{}\label{sect:jac}
Suppose $d=1$. 
In this case we have $\Pic(X, Y, Z)=\Alb(X, Y, Z)$
and we write it as $\Jac(X, Y, Z)$.
We give its geometric description.
Note that $\Jac(X, Y, Z)$ and $\Jac(X, Z, Y)$
are Cartier dual to each other.

In \cite[Definition 25]{IY},
we considered a Laumon $1$-motive  $\LM(X, Y, Z)$.
Explicitely, this is given as follows
(see \cite[\S 5.2]{IY}).
Let $X_Y$ be a proper $\bbC$-curve 
that is obtained by collapsing $Y$ into a single (usually singular) point
(see \cite[Chapter IV, \S 3--4]{MR918564}).
Let $G(X, Y)$ be the generalized Jacobian of $X$ with modulus $Y$
in the sense of Rosenlicht-Serre \cite{MR918564},
or, which amounts to the same,
the Picard scheme $\underline{\Pic}^0(X_Y)$ of $X_Y$.
Let $F_\et(X, Z) =\Div_Z^0(X)$
be the group of degree zero (Cartier) divisors supported on $Z$.
Define $F(X, Z)_{\inf}$ by
$\Lie F(X, Z)_{\inf}=H^0(X, \Osheaf_X(Z-Z_\red)/\Osheaf_X)$.
Put $F(X, Z):=F_\et(X, Z) \times F_\inf(X, Z)$.
Let $u_\et : F_\et(X, Z) \to G(X, Y)$ be the map
that associate to a divisor its isomorphism class.
(We identify $Z$ with its image in $X_Y$.)
Let $u_\inf : F_\inf(X, Z) \to G(X, Y)$ be the map
such that
$\Lie u_\inf$ is given by the composition of
\begin{align*}
&\Lie F_\inf(X, Z) 
= H^0(X, \Osheaf_{X}(Z-Z_\red)/\Osheaf_{X})
\\
&= H^0(X_Y, \Osheaf_{X_Y}(Z-Z_\red)/\Osheaf_{X_Y})
\overset{(*)}{\to} H^1(X_Y, \Osheaf_{X_Y}) = \Lie G(X, Y),
\end{align*}
where $(*)$ is the connecting map with respect to the short exact sequence
\[ 0 \to \Osheaf_{X_Y}\to \Osheaf_{X_Y}(Z-Z_\red) 
\to \Osheaf_{X_Y}(Z-Z_\red)/\Osheaf_{X_Y} \to 0. 
\]
Put $u=u_\et \times u_\inf$, and define
\[ \LM(X, Y, Z)=[u: F(X, Z) \to G(X, Y)]. \]
By \cite[eq. (29)]{IY}, its Deligne part 
$[F_\et(X, Z) \to G(X, Y)_\sa]$
agrees with
Degline's $1$-motive $H^1_m(X_Y \setminus Z)(1)$
from \cite[10.3.4]{MR0498552}.
Here $G(X, Y)_\sa$ denotes the maximal semi-abelian quotient
of $G(X, Y)$.

\begin{prop}
We have $\LM(X, Y, Z) \cong \Jac(X, Y, Z)$.
\end{prop}
\begin{proof}
Let $\eusm L'$
be the object of $\MHSM_1$
corresponding to $\LM(X, Y, Z)$,
and set
$\eusm L=(L, L_\add^\bullet, L_\inf^\bullet, \eusm G^\bullet):=\eusm L'(-1)$.
Let $\eusm H:=\eusm H^1(X, Y, Z)$ be the object 
described in \pararef{sect:curves}
and write $\eusm H=(H, H_\add^\bullet, H_\inf^\bullet, \eusm F^\bullet)$.
It suffices to show that $\eusm L \cong \eusm H$.

We first show that $L \cong H$ as mixed Hodge structures.
Let us consider a commutative diagram
\[
\xymatrix{
Y \ar[r]^i \ar[rd]_-{i'} &
X \ar[r]^-p &
X_Y 
\\
&
X \setminus Z \ar[u]_-j \ar[r]_-{p'} &
U:=X_Y \setminus Z \ar[u]_-{j'},
}
\]
in which 
$i$ and $i'$ are closed immersions,
$j$ and $j'$ are open immersions, and
$p$ and $p'$ are finite morphims.
By applying $Rj_*'$ to an exact sequence
\[ 0 \to \bbZ_U \to p_*' \bbZ_{X \setminus Z} 
\to (p' \circ i')_*\bbZ_Y \to 0, 
\]
we obtain
\begin{align*}
Rj_*' \bbZ_U 
&\cong Rj_*' \Cone(p_*' \bbZ_{X \setminus Z} \to (p' \circ i')_*\bbZ_Y)[-1]
\\
&=\Cone(p_* R_{j*} \bbZ_{X \setminus Z} \to p_* i_ *\bbZ_Y)[-1]
\\
&= p_* \bbZ_{X|Y, Z}.
\end{align*}
It follows that $H^1(U, \bbZ) \cong H^1(X, \bbZ_{X|Y, Z})=H$.
On the other hand, 
we have $L \cong H^1(U, \bbZ)$ 
by \cite[10.3.8]{MR0498552},
whence $L \cong H$.

By definition we have
(see \pararef{sect:Lau-MHSM1-explicit})
\begin{align*}
&H^1_\inf = L_\inf^1 = H^0(X, \Osheaf_X(Z-Z_\red)/\Osheaf_X),
\\
&H^1_\add = H^0(X, \Osheaf_X(-Y_\red)/\Osheaf_X(-Y)),
\\
&L^1_\add = \Lie G(X, Y)_\add,
\end{align*}
where $G(X, Y)_\add$ denotes the additive part of $G(X, Y)$.
It follows $H^1_\add \cong L^1_\add$ by \cite[Lemma 24]{IY}.
In particular we get $\eusm L^1 \cong \eusm H^1$.
Finally, we have
\[ \eusm G^1 = \ker(\eusm L^1 \to \Lie G(X, Y))
\overset{(*)}{\cong} \ker(\eusm H^1 \to H^1(X, \Osheaf_X(-Y))) = \eusm F^1,
\]
where for $(*)$ we used \cite[Chapter V, \S 10, Proposition 5]{MR918564}.
We are done.
\end{proof}

The construction of $\LM(X, Y, Z)$ works 
over any field of characteristic zero.
Thus one can ask the following question:

\begin{ques}
Can $\Pic(X, Y, Z)$ and $\Alb(X, Y, Z)$
be constructed
over any field of characteristic zero when $d>1$?
\end{ques}

When $Z=\emptyset$,
this has been done by Kato and Russell \cite[\S 5]{MR2985516}
(see the next subsection),
and extended by Russell to arbitrary perfect base field \cite{MR3095229}.

\subsection{}
We now consider a smooth projective variety $X$ of dimension $d$
and an effective divisor $D$ on $X$.
Kato and Russell defined in \cite[\S 6.1]{MR2985516}
objects $H^1(X, D_+)(1)$ and $H^{2d-1}(X, D_-)(d)$ 
of their category $\mathcal H_1$
(see \pararef{subsec:KatoRussell}),
and gave their explicit description in  \cite[\S 6.3, 6.4]{MR2985516}.
The Laumon $1$-motive corresponding to $H^{2d-1}(X, D_-)(d)$ 
has trivial formal group part,
that is, it can be written as
$[0 \to \Alb^{KR}(X, D)]$ 
where $\Alb^{KR}(X, D)$ is a commutative algebraic group,
and $\Alb^{KR}(X, D)$ is
Kato-Russell's Albanese variety of $X$ with modulus $D$.

Let us denote by $\eusm H^1(X, D_+)(1)$ and $\eusm H^{2d-1}(X, D_-)(d)$ 
the objects in $\MHSM_1$ that correspond to 
$H^1(X, D_+)(1)$ and $H^{2d-1}(X, D_-)(d)$ 
under the equivalence from Proposition \ref{CompKR}.
Suppose now that $D_\red$ is a simple normal crossing divisor in $X$.
By comparing our construction of $\eusm H^{n}(X, Y, Z)$
with \cite[\S 6.3, 6.4]{MR2985516},
we obtain
\[ 
\eusm H^1(X, D_+)(1) \cong \eusm H^1(X, 0, D)(1)_\fr,
\quad
\eusm H^{2d-1}(X, D_-)(d) \cong \eusm H^{2d-1}(X, D, 0)(d)_\fr.
\]
Therefore we obtain the following.

\begin{prop}
We have $\Alb(X, D, 0) \cong [0 \to \Alb^{KR}(X, D)]$.
\end{prop}

\subsection{}
Lekaus \cite{MR2563147} has defined
Laumon $1$-motives $\Pic_a^+(U)$ and $\Alb_a^+(U)$
for an equidimensional quasi-projective $\bbC$-scheme $U$
of dimension $d$ 
such that its singular locus is proper over $\bbC$.
Their associated Deligne $1$-motives 
agree with the cohomological Picard and Albanese $1$-motives
$\Pic^+(U)$ and $\Alb^+(U)$
constructed by Barbieri-Viale and Srinivas \cite{MR1891270},
hence they correspond to the objects 
$H^1(U, \bbZ)(1)$ and $H^{2d-1}(U, \bbZ)(d)$
of $\MHS_1$ under Deligne's equivalence
$\sM_1^\Del \cong \MHS_1$ from \cite[\S 10]{MR0498552}.
(Lekaus has also defined 
Laumon $1$-motives $\Pic_a^-(U)$ and $\Alb_a^-(U)$
whose associated Deligne $1$-motives
corresponds to the homology groups of $U$.)

We may define objects 
$\eusm H^1(U), \eusm H^{2d-1}(U)$ of $\MHSM$ as follows.
Let $\eusm H^1$ and $\eusm H^{2d-1}$ be 
the objects of $\MHSM_1$ that correspond to 
$\Pic_a^+(U)$ and $\Alb_a^+(U)$
under the equivalence
$\sM_1^\Lau \cong \MHSM_1$ 
from \corollaryref{cor:laumon}.
Then we define
$\eusm H^1(U):=\eusm H^1(-1)$ and 
$\eusm H^{2d-1}(U):=\eusm H^{2d-1}(-d)$.

\begin{ques}
Can the definition of $\eusm H^n(U)$
be extended to $n \not= 1, 2d-1$?
\end{ques}

\begin{rema}
The nature of $\Pic_a^+(U)$ and $\Alb_a^+(U)$ 
are rather different from our 
$\Pic(X, Y, Z)$ and $\Alb(X, Y, Z)$.
For instance, suppose that $U$ is an affine irreducible curve
and let $\ol{U}$ be a good compactification.
Then the Laumon $1$-motive 
$\Pic_a^+(U)=\Alb_a^+(U)=[F_\et \times F_\inf \to G]$
verifies
\[ G=\underline{\Pic}^0(\ol{U}), ~F_\et=\Div^0_{\ol{U} \setminus U}(\ol{U}), ~
\Lie F_\inf= H^1(\ol{U}, \Osheaf_{\ol{U}}).
\]
(In particular, $F_\inf$ depends only on $\ol{U}$, 
as long as $U$ is affine.)
On the other hand, let $X$ be a smooth proper curve
and let $Y, Z$ be effective divisors with disjoint support,
and let $U:=X_Y \setminus Z$ be the curve considered in \pararef{sect:jac}.
Then $\Pic(X, Y, Z)=\Alb(X, Y, Z)=\Jac(X, Y, Z)$
is written as $[F_\et \times F_\inf' \to G]$
using the same $F_\et$ and $G$ as above, but
\[
\Lie F_\inf'= H^0(X, \Osheaf_X(Z-Z_\red)/\Osheaf_X).
\]
\end{rema}

\section{Relation with enriched and formal Hodge structures}
\label{sect:comparison}

\subsection{} \label{sect:EHS}
Let $\tVec_\bbC^\bullet$ be the subcategory of
$\mathbf{Z}^\op\Vec_\bbC$ (see \pararef{sec:def-MHSM})
formed by the objects $V^\bullet$ from \eqref{eq:vec-bul}
such that $V^k$ are trivial for all sufficientally small $k$
and such that $\tau_V^k$ are isomorphic for all sufficientally large $k$.
For an object $V^\bullet$ of $\tVec_\bbC^\bullet$,
we denote by $V^\infty$ the projective limit of $(V^k, \tau_V^k)_{k \in \bbZ}$
and by $\tau_V^{\infty, k} : V^\infty \to V^k$ the canonical map.
For a mixed Hodge structure $H$,
we define an object of $\tVec_\bbC^\bullet$ by
\[
H_\bbC/F^\bullet 
:= (\dots \to H_\bbC/F^k H_\bbC \to H_\bbC/F^{k-1} H_\bbC \to \dots),
\]
where all maps are the projection maps.
We have $H_\bbC/F^\infty = H_\bbC$.

Recall from \cite{MR1940668} that
an enriched Hodge structure is a tuple
$E=(H, V^\bullet, v^\bullet, s)$
of 
a mixed Hodge structure $H$,
an object $V^\bullet$ of $\tVec_\bbC^\bullet$,
a morphism 
$v^\bullet : V^\bullet \to H_\bbC/F^\bullet$ of $\tVec_\bbC^\bullet$,
and a $\bbC$-linear map $s : H_\bbC \to V^\infty$
such that $v^\infty \circ s = \id$.
A morphism between two enriched Hodge structures is 
a pair of morphisms of $\MHS$ and of $\tVec_\bbC^\bullet$
that is compatible with structural maps $(v^\bullet, s)$.
The category of enriched Hodge structures is denoted by $\EHS$.
Let $\EHS_\triangle$ be the full subcategory of $\EHS$
consisting of objects $(H, V^\bullet, v^\bullet, s)$
such that $v^k$ are isomorphic for all sufficientally large $k$
(hence $s=(v^\infty)^{-1}$).
Recall from \sectionref{sec:h-add}
that we have defined a subcategory $\MHSM_\add$
of $\MHSM$.

\begin{lemm}\label{lem:EHS-MHSMadd}
The categories $\EHS_\triangle$ and $\MHSM_\add$ are equivalent.
\end{lemm}
\begin{proof}
Take an object $(H, V^\bullet, v^\bullet, s)$ of $\EHS_\triangle$.
We define an object $(H, H_\add^\bullet, 0, \eusm F^\bullet)$ of $\MHSM_\add$
by setting for each integer $k$
\[ H_\add^k := \ker(v^k), \quad
 \eusm F^k := \ker(H_\bbC \oplus \ker(v^k) \to V^k),
\]
where the last map is defined by 
$\tau_V^{\infty, k} \circ s : H_\bbC  \to V^k$
and the inclusion map $\ker(v^k) \hookrightarrow V^k$.
This yields a functor $\EHS_\triangle \to \MHSM_\add$.
Next, 
take an object $(H, H_\add^\bullet, 0, \eusm F^\bullet)$ of $\MHSM_\add$.
We defin an object 
$(H, V^\bullet, v^\bullet, s)$ of $\EHS_\triangle$
by setting for each integer $k$
\[ V^k := (H_\bbC \oplus H_\add^k)/\eusm F^k,
\quad
v^k : V^k \to H_\bbC/F^k H_\bbC,
\quad
s := (v^{\infty})^{-1},
\]
where $v^k$ is induced by 
the composition of the projection maps
$H_\bbC \oplus H_\add^k \to H_\bbC \to H_\bbC/F^k H_\bbC$.
This yields a functor $\EHS_\triangle \to \MHSM_\add$.
It is easy to see that these two functors are quasi-inverse
to each other.
\end{proof}

\begin{rema}
One can construct a category that contains both
$\EHS$ and $\MHSM$ as follows.
Let $\tVec_\bbC^{\bullet \vee}$ be the subcategory of
$\mathbf{Z}^\op\Vec_\bbC$
formed by the objects $V^\bullet$ 
such that $V^k$ are trivial for all sufficientally large $k$
and such that $\tau_V^k$ are isomorphic for all sufficientally small $k$.
We define $\widetilde{\MHSM}$ to be the category of tuples
$(H,H^\bullet_\iff,H^\bullet_\add,\eusm F^\bullet)$
consisting of 
an object $(H,H^\bullet_\iff,H^\bullet_\add)$ of
$\MHS \times \tVec_\bbC^\bullet \times \tVec_\bbC^{\bullet \vee}$
and a linear subspace $\matheusm F^k$ of 
$\eusm H^k = H_\bbC \oplus H_\iff^k \oplus H_\add^k$ 
for each $k \in \bbZ$,
subject to the conditions
{\bf{(\ref{defi:MHSM}-a})}-{\bf{(\ref{defi:MHSM}-d})}
in \definitionref{defi:MHSM}.
Then $\EHS$ is identified with a subcategory 
$\widetilde{\MHSM}_\add$ of $\widetilde{\MHSM}$ 
consisting of objects
such that $H_\inf^\bullet=0$.
Note that the functor $R$ from \pararef{sect:functor-R}
cannot be extended to $\widetilde{\MHSM}$.
Note also that 
$\tVec_\bbC^\bullet$ 
(resp. $\tVec_\bbC^{\bullet \vee}$)
is not Noetherian (resp. Artinian),
hence $\widetilde{\MHSM} \otimes \bbQ$
is neither Noetherian nor Artinian,
while $\Vec_\bbC^\bullet$ and $\MHSM \otimes \bbQ$
(as well as $\MHS \otimes \bbQ$) are both Artinian and Noetherian.
\end{rema}

\subsection{}
Let $n$ be a positive integer.
We write $\MHS^n$ for the subcategory of $\MHS$
consisting of mixed Hodge structures $H$ 
such that 
$\Gr_F^p \Gr^W_{p+q} H_\bbC = 0$
unless $p, q \in [0, n]$.
Denote by $\EHS^n$ the full subcategory of $\EHS$
consisting of
objects $E=(H, V^\bullet, v^\bullet, s)$
such that $H_\bbC$ belongs to $\MHS^n$,
$V^k=0$ for any $k \le 0$,
and $\tau_V^{\infty, k}$ are isomorphic for any $k>n$.
Let $\EHS^n_\triangle$ be the intersection of 
$\EHS^n$ and $\EHS_\triangle$.
We define a functor $\triangle_n : \EHS^n \to \EHS^n_\triangle$ by
$\triangle_n(H, V^\bullet, v^\bullet, s) 
=(H, V_\triangle^\bullet, v_\triangle^\bullet, \id)$,
where
\[
V_\triangle^k:=
\begin{cases}
V^k & \text{if}~ k \le n \\
H_\bbC & \text{if}~  k > n,
\end{cases}
\quad
\tau_{V_\triangle}^k:=
\begin{cases}
\tau_V^k & \text{if}~ k \le n \\
\tau_V^{\infty, n} \circ s & \text{if}~ k =n+1 \\
\id_{H_\bbC} & \text{if}~  k > n+1,
\end{cases}
\quad
v_\triangle^k:=
\begin{cases}
v^k & \text{if}~ k \le n \\
\id_{H_\bbC} & \text{if}~ k > n.
\end{cases}
\]

According to Mazzari \cite{MR2782612},
a formal Hodge structure of level $\le n$ is
a tuple $(E, U, u)$ of 
an object $E=(H, V^\bullet, v^\bullet, (v^\infty)^{-1})$
of $\EHS_\triangle^n$,
a finite dimensional $\bbC$-vector space $U$,
and a $\bbC$-linear map $u : U \to V^n$.
The category of formal Hodge structures of level $\le n$
is denoted by $\FHS^n$.
We identify the subcategory of $\FHS^n$
formed by objects of the form $(E, 0, 0)$
with $\EHS^n_\triangle$.

Let $\MHSM^n_\square$ be a full subcategory of $\MHSM$
consisting of objects 
$\eusm H=(H, H_\add^\bullet, H_\inf^\bullet, \eusm F^\bullet)$
such that $H$ belongs to $\MHS^n$,
$H_\add^k=0$ unless $k \in [1, n]$.
and $H_\inf^k=0$ for all $k \not= n$.
Let 
$\MHSM_\diamond^n$ be the full subcategory of 
$\MHSM_\square^n$ 
consisting of objects 
$\eusm H=(H, H_\add^\bullet, H_\inf^\bullet, \eusm F^\bullet)$
such that the composition map
$H_\inf^n \hookrightarrow \eusm H_\inf^n 
\twoheadrightarrow \eusm H_\inf^n/\eusm F_\inf^n$
is the zero map.

\begin{lemm}\label{lem:FHS-MHSM}
\begin{enumerate}
\item 
The categories $\FHS^n$ and $\MHSM_\square^n$ are equivalent.
\item 
The categories $\EHS^n$ and $\MHSM_\diamond^n$ are equivalent.
\end{enumerate}
\end{lemm}
\begin{proof}
(1) Take an object $((H, V^\bullet, v^\bullet, s), U, u)$ of $\FHS^n$.
We define an object $(H, H_\add^\bullet, H_\inf^\bullet, \eusm F^\bullet)$ 
of $\MHSM_\square^n$
by setting for each integer $k$
\begin{align*}
& H_\add^k := \ker(v^k),
\qquad
H_\inf^k := 
\begin{cases}
U & k = n \\
0 & k \not= n,
\end{cases}
\\
& \eusm F^k := 
\begin{cases}
\ker(H_\bbC \oplus \ker(v^n) \oplus U \to V^n)
& k = n \\
\ker(H_\bbC \oplus \ker(v^k)  \to V^k),
& k \not= n,
\end{cases}
\end{align*}
where the last map is defined by 
$\tau_V^{\infty, k} \circ s : H_\bbC  \to V^k$,
the inclusion map $\ker(v^k) \hookrightarrow V^k$,
and $u : U \to V^n$.
This yields a functor $\FHS^n \to \MHSM_\square^n$.
Next, 
take an object 
$\eusm H=(H, H_\add^\bullet, H_\inf^\bullet, \eusm F^\bullet)$ 
of $\MHSM_\square^n$.
Let $E$ be an enriched Hodge structure
that corresponds to $\eusm H_\add=\pi_\add(\eusm H)$
(see \eqref{eq:h_add})
under the equivalence in \lemmaref{lem:EHS-MHSMadd}.
Then $E=(H, V^\bullet, v^\bullet, s)$ belongs to $\EHS^n_\triangle$,
and we have $V^k=\eusm H_\add^k/\eusm F_\add^k$.
Set $U=H_\inf^n$.
We define a linear map
$u : H_\inf^n \to \eusm H^n_\add/\eusm F^n_\add$
as the composition of
\[ H_\inf^n \hookrightarrow \eusm H^n
 \twoheadrightarrow \eusm H^n/\eusm F^n
\overset{\cong}{\leftarrow} \eusm H^n_\add/\eusm F^n_\add,
\]
where the last isomorphism is from \eqref{eq:fund-exseq2}.
We have defined an object $(E, U, u)$ of $\FHS^n$.
This yields a functor $\MHSM_\square^n \to \FHS^n$.
It is easy to see that these two functors are quasi-inverse
to each other, proving (1).

(2) There is a full faithful functor
$\sigma_n : \EHS^n \to \FHS^n$ given by
\[ E=(H, V^\bullet, v^\bullet, s)
\mapsto (\triangle_n(E), \ker(v^\infty), u),
\]
where $u$ is the composition of the inclusion map
$\ker(v^\infty) \hookrightarrow V^\infty$
and $\tau_V^{\infty, n} : V^\infty \to V^n$.
Its essential image is formed by 
objects $((H, V^\bullet, v^\bullet, s), U, u)$ of $\FHS^n$
such that $v^n \circ u = 0$.
(See \cite[Proposition 4.2.3]{MR2554937}
for the case $n=1$.
Formal Hodge structures satisfying the last condition 
are called special in \cite{MR2782612})
Under the equivalence from the first part of the lemma,
the last condition is translated to
$H_\inf^n \to \eusm H_\inf^n/\eusm F_\inf^n$
being the zero map by \eqref{eq:map-hf}.
\end{proof}


\begin{lemm}\label{lem:ext-FMS-MHSM}
Denote by $\iota_n : \FHS^n \to \MHSM$
the composition of the equivalence functor
$\FHS^n \cong \MHSM^n_\square$ from \lemmaref{lem:FHS-MHSM}
and the inclusion functor $\MHSM^n_\square \subset \MHSM^n$.
Then, 
for any objects $D, D'$ of $\FHS^n$, we have
\[ \Ext^1_{\FHS^n}(D, D') \cong \Ext^1_{\MHSM}(\iota_nD, \iota_nD'). \]
\end{lemm}
\begin{proof}
By \lemmaref{lem:FHS-MHSM},
it suffices to show 
$\MHSM^n_\square$ is a thick Abelian subcategory of $\MHSM$.
This follows from \remarkref{rem:exact}.
\end{proof}

\subsection{}
Let $n$ be an integer and
let $X$ be a proper irreducible variety over $\bbC$ of dimension $d$.
Bloch and Srinivas constructed 
an enriched Hodge structure 
$H^n_\BS(X)=(H, V^\bullet, v^\bullet, s)$ as follows. 
(There are variants, see \cite[Corollary 2.2]{MR1940668}.)
Let $H=H^n(X, \bbZ)$ be Deligne's mixed Hodge structure.
Take a smooth (proper) hypercovering $\pi : X_* \to X$.
Then we have 
$H^n(X_*, \Omega_{X_*}^{\bullet <k})=H_\bbC/F^k H_\bbC$,
and there is a commutative diagram
\[
\xymatrix{
&
H^n(X_*, \bbC) \ar[r]^{\cong}
&
H^n(X_*, \Omega_{X_*}^\bullet) \ar[r] 
&
H^n(X_*, \Omega_{X_*}^{\bullet<k})
\\
H_\bbC \ar[r]^{=} 
&
H^n(X, \bbC) \ar[r]^s \ar[u]^{\cong}
&
H^n(X, \Omega_{X}^{\bullet<d+1}) \ar[r]  \ar[u]^{v^\infty}
&
H^n(X, \Omega_{X}^{\bullet<k}),  \ar[u]^{v^k}.
}
\]
We define an object $H^n_\BS(X)$ of $\EHS^n$
by setting 
$V^k:=H^n(X, \Omega_{X}^{\bullet<k})$ if $k \le d$
and
$V^k:=H^n(X, \Omega_{X}^{\bullet<d+1})$ if $k > d$.
This belongs to $\EHS^d$ if $n \ge d$.

In \cite[Definition 3.1]{MR2782612},
Mazzari defined a similar objects $H^{n, k}_\sharp(X)$
of $\EHS^n_\triangle$ for each $k=1, \dots, n$.
Let us have a look at a special case $H^{2d-1, d}_\sharp(X)$
from \cite[Example 3.2]{MR2782612},
which actually belongs to $\EHS^d_\triangle$
and the same as $\triangle_d(H^{2d-1}_\BS(X))$.
We write $\eusm H_\triangle^{2d-1}(X)$ 
for the corresponding object of $\MHSM$
under the equivalence from \lemmaref{lem:FHS-MHSM}.
By \lemmaref{lem:ext-FMS-MHSM},
Mazzari's result \cite[Proposition 3.5]{MR2782612}
can be rewritten as follows:

\begin{prop}
Let $X$ be an irreducible proper variety over $\bbC$ of dimension $d$.
Denote by $A^d(X)$ the Albanese variety of $X$
in the sense of Esnault-Srinivas-Viehweg \cite{MR1669284}.
Then there is an isomorphism
\[ A^d(X) \cong \Ext_{\MHSM}^1(\bbZ(-d), \eusm H_\triangle^{2d-1}(X)).
\]
\end{prop}


\end{document}